\documentclass[12pt]{amsart}

\usepackage{graphicx}

\newcommand{\ilimit}{\,\varprojlim{}\!}
\newcommand{\dlimit}{\,\varinjlim{}\!}
\newcommand{\mono}{\rightarrowtail}
\newcommand{\epi}{\twoheadrightarrow}

\usepackage[arrow,matrix,curve]{xy}\SilentMatrices
\def\xyma{\xymatrix@M.7em}

\newcommand{\para}{\par\vspace{.25cm}}
\newtheorem{Lemma}{Lemma}[section]
\newtheorem{Proposition}[Lemma]{Proposition}
\newtheorem{Theorem}[Lemma]{Theorem}
\newtheorem{Corollary}[Lemma]{Corollary}
\newtheorem{prop}[Lemma]{Proposition}
\newtheorem{Conjecture}[Lemma]{Conjecture}
\newtheorem{Remark}[Lemma]{Remark}
\usepackage{amssymb}
\usepackage{amsfonts}
\usepackage{euscript}

\numberwithin{equation}{section}

\begin{document}

\title{A higher limit approach to homology theories}
\author{Sergei O. Ivanov}
\address{Chebyshev Laboratory, St. Petersburg State University, 14th Line, 29b,
Saint Petersburg, 199178 Russia} \email{ivanov.s.o.1986@gmail.com}

\author{Roman Mikhailov}
\address{Chebyshev Laboratory, St. Petersburg State University, 14th Line, 29b,
Saint Petersburg, 199178 Russia and St. Petersburg Department of
Steklov Mathematical Institute} \email{rmikhailov@mail.ru}
\urladdr{http://www.mi.ras.ru/\~{}romanvm/pub.html}

\thanks{This research is supported by the Chebyshev
Laboratory  (Department of Mathematics and Mechanics, St.
Petersburg State University)  under RF Government grant
11.G34.31.0026 and by JSC "Gazprom Neft". The first author is
supported by RFBR (grant no. 12-01-31100 mol\_a, 13-01-00902 A)}

\begin{abstract}
A lot of well-known functors such as group homology, cyclic
homology of algebras can be described as limits of certain simply
defined functors over categories of presentations. In this paper,
we develop technique for the description of the higher limits over
categories of presentations and show that certain homological
functors can be described in this way. In particular, we give a
description of Hochschild homology and the derived functors of tensor, symmetric and
exterior powers in the sense of Dold and Puppe as higher limits.
\end{abstract}
\maketitle

\section{Introduction}

Let $k$ be a ring, ${\sf Mod}(k)$ the category of $k$-modules and
let $\mathcal C$ be a category such that for any two objects $c$
and $c'$ there exists a morphism $f:c\to c'.$ Then for a functor
$\mathcal F:\mathcal C\to {\sf Mod}(k)$ the limit  $\ilimit\
\mathcal F$ is the largest constant subfunctor in $\mathcal F.$
Therefore, one can perceive $\ilimit\ \mathcal F$ as the largest
part of $\mathcal F(c)$ which does not dependent on $c$.

Let $A$ be an ``algebraic object'' (group, abelian group, associative algebra, ...). We denote by  ${\sf Pres}(A)$ the category of presentations of
$A$ as a quotient of a free object $F\epi A$.
 This category satisfies the condition above and a lot of interesting ``homology theories of $A$'' can be described as limits of simple functors
  ${\sf Pres}(A)\to {\sf Mod}(k)$ without use of  homological algebra.
D. Quillen in \cite{Q} proved that if $A$ is an algebra over a field of characteristic zero, then even cyclic homology can be described as the limit
\begin{equation}\label{int_HC}
HC_{2n}(A)=\ilimit\ F/(I^{n+1}+[F,F]).
\end{equation}
The limit is taken over the category of presentations $I\mono F\epi A,$ where $F$ is a free algebra. In other words,  $HC_{2n}(A)$ is the largest part of $F/(I^n+[F,F])$ independent of the presentation $I\mono F\epi A.$ He also proved the similar formula for the odd reduced cyclic homology:
\begin{equation}\label{int_HC_red} \overline{HC}_{2n+1}(A)=\ilimit \ I^{n+1}/[I,I^n].
\end{equation}
I. Emmanouil and R. Mikhailov \cite{EM} proved the similar formulas for the case of group homology:
\begin{equation}\label{int_H}
H_{2n}(G,M)=\ilimit\  R_{ab}^{\otimes n}\otimes_{\mathbb Z[G]}M.
\end{equation}
The limit is taken over the category of presentations $R\mono
F\epi G$, where $F$ is a free group and $R_{ab}$ is the relation
module. Further R. Mikhailov and I. B. S. Passi \cite{MP}
described the highest nonzero Dold-Puppe derived functors of
symmetric, exterior and tensor powers as limits in the category of
free presentations of an abelian group. In this article we prove
the following formula for even Hochschild homology of an algebra
$A$ over a field of any characteristic:
\begin{equation}\label{int_HH}
H_{2n}(A,M)= \ilimit \ (I^n/I^{n+1})\otimes_{A^e} M,
\end{equation}
where $A^e=A\otimes A^{\rm op}.$

The most exiting thing in this approach is the fact that a lot of
homology theories can be described without any recourse to
homological algebra but using only presentations. Moreover, thanks
to the interpretation of a limit as the largest part independent
of presentation this description is very intuitive. On the other
hand, this approach shows that homologies can be useful in work
with an algebraic object by means of its presentation. The
additional advantage of this approach is that it gives a way to
observe different maps between homology theories.  If we denote
$M_\natural=M/[F,M]=H_0(F,M),$ then the formulas  \eqref{int_HC}
and \eqref{int_HH} imply  \begin{equation} HC_{2n}(A)=\ilimit\
(F/I^{n+1})_\natural \ \ \text{ и }\ \  HH_{2n}(A)=\ilimit\
(I^n/I^{n+1})_\natural
\end{equation} for algebras over a field of characteristic zero.
Therefore, the short exact sequence  $I^n/I^{n+1}\mono F/I^{n+1} \epi F/I^n$ gives a sequence of homomorphisms $HH_{2n}(A)\to HC_{2n}(A)\to HC_{2n-2}(A).$

However, this approach has some disadvantage. The main of them is that the homologies are described as limits only for a part of indexes. Another disadvantage is impossibility to obtain long exact sequences for these homologies using this language.

In this work we improve these drawbacks using the higher limits. A
representation of a category $\mathcal C$ over a commutative ring
$k$ is a functor $\mathcal F:\mathcal C\to {\sf Mod}(k).$ The
category of representations of  $\mathcal C$ is denoted by ${\sf
Mod}(k)^{\mathcal C}.$ Then the limit of a representation is a
left exact functor  between abelian categories $\ilimit : {\sf
Mod}(k)^{\mathcal C}\to {\sf Mod}(k)$. Its derived functors are
called higher limits and denoted by $\ilimit^i:={\bf
R}^i\!\ilimit.$ The main aim of the article is to generalise the
formulas  \eqref{int_H}-\eqref{int_HH} to the following formulas
(see theorems \ref{groupisomorphism} and \ref{hoshc})
\begin{equation}
H_{2n-i}(G,M)=\ilimit^i\  R^{\otimes n}_{ab}\otimes_{\mathbb Z[G]}
M,\ \  \text{ for } i<n,
\end{equation}
\begin{equation}
H_{2n-i}(A,M)=\ilimit^i\  (I^n/I^{n+1})\otimes_{A^e} M,\ \  \text{ for } i<n,
\end{equation}
and give the following description of the derived functors of
symmetric, exterior and tensor powers (see theorem
\ref{derivedth})
\begin{equation} L_{n-i} S^n(A)=\ilimit^i\ \Lambda^n (H),
\end{equation}
\begin{equation}
L_{n-i} \Lambda^n(A)=\ilimit^i\ \Gamma^n (H),\end{equation}
\begin{equation}L_{n-i} \otimes^n(A)=\ilimit^i\ {\otimes^n}(H). \end{equation}
The limits are taken over the category of presentations  $H\mono
F\epi A,$ where $F$ is a free abelian group. The formulas
\eqref{int_HC}, \eqref{int_HC_red} are generalised for the case of
augmented algebras over a field of characteristic zero to the
folowing formulas (see theorem \ref{cychomlim})
\begin{equation}\label{int_HC_higher}
HC_{2n-i}(A)=\ilimit^i\ F/(I^{n+1}+[F,F]), \ \ \ \text{ for } i\in \{0,1\},
\end{equation}
\begin{equation}\label{int_HC_red_higher} \overline{HC}_{2n+1-i}(A)=\ilimit^i \ I^{n+1}/[I,I^n], \ \ \ \text{ for } i\in \{0,1\}.
\end{equation}
These formulas describe the homology theories for all indexes. It should be noted that any commutative algebra $A$ of finite type over the field $\mathbb{C}$ has an augmentation $A\to \mathbb C.$ Thus the formulas  \eqref{int_HC_higher} and \eqref{int_HC_red_higher} hold for all coordinate algebras of affine complex varieties.

Further, we give another point of view on group homology through
the higher limits. Let  $R\hookrightarrow F \overset{\pi}\epi G$
be a presentation of a group $G$. Denote by $F_n(\pi)$ the $n$th
fibred power  $F\times_G  \cdots \times_G F.$ Then there are
recurrence relations, which express higher homologies through the
first homologies (see theorem \ref{fibretheorem})
\begin{equation}
H_{2n-2}(G,M)=\ilimit\ H_{n-1}(F_n(\pi),M) \ \ \text{ for }  n\geq 2,
\end{equation}
\begin{equation}
H_{2n-3}(G,M)=\ilimit^1\ H_{n-1}(F_n(\pi),M) \ \ \text{ for } \ n\geq 3.
\end{equation}
Furthermore, we compute limits of some higher relation modules. For example, if  $G$ is a 2-torsion-free group then
\begin{equation}
\ilimit\ [R,R]/[R,R,F] = H_4(G,\mathbb Z/2),
\end{equation}
\begin{equation}
\ilimit^1\ [R,R]/[R,R,F] = H_3(G,\mathbb Z/2).
\end{equation}

The paper is organized as follows. In section \ref{section_Representations_of_categories} we recall the theory of representations of categories  \cite{Xu}, \cite{Jackowski-Slominska},
\cite{Webb}. In section \ref{section_Representations_of_categories_with_pairwise_coproducts} we develop technique of computing of higher limits of representations of categories with pairwise coproducts. The main result in this section is proposition \ref{proposition_cokernel_limits} which allows to compute the higher limits for so-called monoadditive functors. Section \ref{section_4-term_exact_sequences} is devoted to  4-term exact sequences of the form ${\mathcal H}_*(A)\mono \mathcal F_1 \to \mathcal F_2 \epi \mathcal H_{*-1}(A)$ where $\mathcal F_i$ are  representations of ${\sf Pres}(A)$ and $\mathcal H_*(A)$ is a one of the homology theories of an algebraic object $A$ (considered as a constant functor). These sequences are very useful for computing of higher limits. Sections \ref{section_Group_homology_as_higher_limits}, \ref{section_Hochschild_and_Cyclic_homology},
\ref{section_Derived_functors_in_the_sense_of_Dold-Puppe_as_higher_limits} are devoted to higher limit formulas of homology theories for groups, algebras and abelian groups respectively, and in section \ref{section_Hihger_relation_modules} we compute higher limits for some higher relation modules over groups.

\section{Representations of categories and higher limits.}\label{section_Representations_of_categories}
In this section  we recall several facts about
representations of categories and higher limits and colimits which
the reader can find in \cite{Xu}, \cite{Jackowski-Slominska},
\cite{Webb} with more details and proofs (see also
\cite{Quillen}).

\subsection{Representations of categories.}
 Let $k$ be a commutative ring and $\mathcal{C}$ be a category. Using considerations of Grothendieck universes we assume that all categories are small for an appropriate universe. By definition, a {\it representation} of $\mathcal{C}$ (over $k$) is
 a functor $\mathcal{F}:\mathcal{C}\to \mathsf{Mod}(k).$ The category of representations $\mathsf{Mod}(k)^\mathcal{C}$ is an abelian category with enough projectives and injectives,
  and a sequence in this category is exact if and only if it is exact objectwise. Denote by $k\mathcal{C}$ the category algebra of $\mathcal{C}$
   i.e. the algebra with the basis $\mathsf{Mor}(\mathcal{C})$ whose multiplication is induced by the composition and vanishes for non-composable
    morphisms. This algebra has an identity element if and only if $\mathcal{C}$ has finitely many objects. The category of representations is
     naturally embedded into the category of (left) $k\mathcal{C}$-modules as an abelian subcategory:
$$\mathsf{Mod}(k)^\mathcal{C}\hookrightarrow \mathsf{Mod}(k\mathcal{C}),\ \ \mathcal{F}\mapsto \bigoplus_{c\in \mathcal{C}} \mathcal{F}(c),$$
and its image consists of $k\mathcal{C}$-modules $M$ such that
$M=\bigoplus_{c\in \mathcal{C}} \mathsf{id}_c\cdot M$. We will
identify a representation of $\mathcal{C}$ and the corresponding
$k\mathcal{C}$-module. All these definitions  extend the
corresponding notions for groups considered as a category with one
object and for quivers which are replaced by the associated free
categories.

\subsection{Higher limits and cohomology.}Consider the diagonal
functor $\mathsf{Mod}(k)\to \mathsf{Mod}(k)^\mathcal{C}$ which
sends a $k$-module $M$ to a constant functor $\mathcal{C}\to
\mathsf{Mod}(k)$ sending all objects to $M$ and all morphisms to
$\mathsf{id}_M.$ We denote this {\it constant representation} by
the same symbol $M:\mathcal{C}\to \mathsf{Mod}(k).$ Then the limit
and colimit functors $\ilimit, \dlimit:
\mathsf{Mod}(k)^\mathcal{C}\to \mathsf{Mod}(k)$ are defined as the
right and the left adjoint functors to the diagonal functor.
Therefore, we have the natural isomorphisms:
\begin{equation} \mathsf{Hom}_{k\mathcal{C}}(M,\mathcal{F})\cong \mathsf{Hom}_k(M,\ilimit\, \mathcal{F}), \ \ \ \
\mathsf{Hom}_{k\mathcal{C}}(\mathcal{F},M)\cong
\mathsf{Hom}_k(\dlimit\, \mathcal{F},M),
\end{equation} for any representation
$\mathcal{F}$ and any $k$-module $M.$ Since $\ilimit$ is a right
adjoint functor, it is a left exact additive functor. Hence it has
right derived functors, and similarly $\dlimit$ has left derived
functors.
\begin{equation}\ilimit^i={\bf R}^i\!\ilimit, \ \ \dlimit_i={\bf L}_i\!\dlimit.
\end{equation}
The $n$th homology and cohomology group of $\mathcal{C}$ with
coefficients in a representation $\mathcal{F}$ are defined as
follows:
\begin{equation}
H_n(\mathcal{C},\mathcal{F})=\mathsf{Tor}^{k\mathcal{C}}_n(k,\mathcal{F}), \ \ H^n(\mathcal{C},\mathcal{F})=\mathsf{Ext}_{k\mathcal{C}}^n(k,\mathcal{F}).
\end{equation}
For any $n\geq 0$ there are isomorphisms
\begin{equation}\ilimit^n\ \mathcal{F}\cong H^n(\mathcal{C},\mathcal{F}), \ \ \dlimit^n\ \mathcal F\cong H_n(\mathcal{C},\mathcal{F}).
\end{equation}

\subsection{Preserving of higher limits.} Recall that the {\it
nerve} of a category $\mathcal{C}$ is a simplicial set
$\mathsf{N}\mathcal{C}:\Delta^{op}\to \mathsf{Sets}$ given by the
formula $\mathsf{N}\mathcal{C}=\mathsf{Cat}(-,\mathcal{C})$ (here
we consider $\Delta$ as a full subcategory of $\mathsf{Cat}$ where
$[n]$ is a free category generated by the graph $0\to 1\to \dots
\to n$), and the classifying space of the category $\mathcal{C}$
is the geometrical realization of its nerve
$\mathsf{B}\mathcal{C}=|\mathsf{N}\mathcal{C}|.$ A functor
$\Phi:\mathcal{C}\to \mathcal{D}$ induces a morphism of simplicial
sets $\mathsf{N}\Phi:\mathsf{N}\mathcal{C}\to
\mathsf{N}\mathcal{D}$ and a continuous map
$\mathsf{B}\Phi:\mathsf{B}\mathcal{C}\to \mathsf{B}\mathcal{D}.$
Moreover, a natural transformation $\alpha:\Phi_1\to \Phi_2$
induces a homotopy $\mathsf{B}\alpha$ between $\mathsf{B}\Phi_1$
and $\mathsf{B}\Phi_2$ \cite[\S 1 proposition 2]{Quillen}.  (One
could say that $\mathsf{B}$ is a strict 2-functor.)

For a $k$-module $M$ the (co)homology groups of $\mathcal{C}$ with
coefficients in the constant representation $M$ are isomorphic to
(co)homology groups of the classifying space
$\mathsf{B}\mathcal{C}$ with coefficients in $M:$
\begin{equation}H_n(\mathcal{C},M)=H_n(\mathsf{B}\mathcal{C},M), \ \ \ H^n(\mathcal{C},M)=H^n(\mathsf{B}\mathcal{C},M).\end{equation}
A category $\mathcal{C}$ is said to be {\it contractible} if its
classifying space $\mathsf{B}\mathcal{C}$ is contractible and it
is said to be $k${\it -acyclic}
 if the reduced homology groups $\tilde H_*(\mathsf{B}\mathcal{C},k)$ vanish. Of course, a contractible category is $k$-acyclic.

Let  $\Phi:\mathcal{C}\to \mathcal{D}$ be a functor. For an object
$d\in D$ we denote by $\Phi\!\!\downarrow\!\!d$ the comma category
(see \cite{Mac_Lane}) consisting of objects $(c,\alpha)$, where
$c\in \mathcal{C}$  and $\alpha\in \mathcal{D}(\Phi(c),d)$. A
morphism in the comma category from  $(c,\alpha)$ to
$(c',\alpha')$ is given by $f:c\to c'$, which satisfies $\alpha
\Phi(f)=\alpha'.$ If $\mathcal{C}=\mathcal{D}$ and
$\Phi=\mathsf{Id}_\mathcal{D}$, the corresponding comma category
is denoted by $\mathcal{D}\!\!\downarrow\!\!d.$

\begin{Proposition}\label{proposition_A}{\rm (see \cite[5.4]{Jackowski-Slominska}).}
Let $\Phi:\mathcal{C}\to \mathcal{D}$ be a functor satisfying the
condition that $\Phi\!\!\downarrow\!\!d$ is $k$-acyclic for any
$d\in D.$ Then for any representation $\mathcal{F}$  of
$\mathcal{D}$ there is an isomorphism
\begin{equation}\theta_{\Phi}:\ilimit^n\mathcal{F}
\overset{\cong}{\longrightarrow}  \ilimit^n(\mathcal{F} \Phi),\end{equation}
which is natural by $\mathcal{F}$ and $\Phi.$
\end{Proposition}
\begin{Remark}
In the previous proposition by the naturalness by $\Phi$ we mean
that for a natural transformation $\alpha:\Phi_1\to \Phi_2$ of
functors satisfying the condition of the proposition we have
$\alpha_*\theta_{\Phi_1}=\theta_{\Phi_2},$ where
$\alpha_*:\ilimit^n(\mathcal{F}\Phi_1)\to
\ilimit^n(\mathcal{F}\Phi_2)$ is the map induced by $\alpha.$ In
\cite{Jackowski-Slominska} the fact that the isomorphism is
natural did not state, but it follows easily from the proof and
the fact that the obvious map $\theta^0_\Phi:\ilimit\:\mathcal{F}\to
\ilimit(\mathcal{F} \Phi)$ is natural by $\mathcal{F}$ and $\Phi.$
\end{Remark}
\begin{Corollary}\label{corollary_A_naturalness}
Let $\Phi_1,\Phi_2:\mathcal{C}\to \mathcal{D}$ be functors
satisfying the condition of proposition \ref{proposition_A} and
$\alpha:\Phi_1\to \Phi_2$ is a natural transformation. The induced
homomorphism $\alpha_*:\ilimit^n(\mathcal{F}\Phi_1)\to
\ilimit^n(\mathcal{F}\Phi_2)$ is an isomorphism.
\end{Corollary}
\begin{proof}
Since the isomorphism in proposition \ref{proposition_A} is
natural by $\Phi,$ there is a commutative diagram
$$\xymatrix{
 && \ilimit^n\mathcal{F}\ar[dll]_{\theta_{\Phi_1}}^\cong\ar[drr]^{\theta_{\Phi_2}}_\cong  && \\
\ilimit^n(\mathcal{F}\Phi_1)\ar[rrrr]^{\alpha_*} && &&
\ilimit^n(\mathcal{F}\Phi_2). }$$ Hence, $\alpha_*$ is an
isomorphism.
\end{proof}
\begin{Lemma}\label{lemma_adjoint}
Left adjoint functors satisfy the condition of proposition
\ref{proposition_A}. In particular, if $\Phi:\mathcal{C}\to
\mathcal{D}$ is a left adjoint functor and $\mathcal{F}$ is a
representation of $\mathcal{D},$ then there is an isomorphism
$\ilimit^n\mathcal{F}\cong \ilimit^n(\mathcal{F} \Phi)$ for any
$n\geq 0.$
\end{Lemma}
\begin{proof}
Let $\Psi:\mathcal{D}\to \mathcal{C}$ be a right adjoint functor
to $\Phi.$ Then the comma category $\Phi\!\!\downarrow\!\!d$ has a
terminal object which is given by the the pair
$(\Psi(d),\varepsilon_d),$ where $\varepsilon_d:\Phi\Psi(d)\to d$
is the counit of adjunction (\cite[IV.1]{Mac_Lane}). Thus the
category $\Phi\!\!\downarrow\!\!d$ is contractible (\cite[\S 1
Corollary 2]{Quillen}).
\end{proof}

\subsection{Spectral sequence of higher limits.}
\begin{Proposition} Let $\Phi:\mathcal A\to \mathcal B$ be a left exact functor between abelian categories with enough injectives and $a^\bullet\in {\sf Com}^+(\mathcal A)$ be a bounded below complex with $\Phi$-acyclic cohomologies.
Then there exists a cohomological spectral sequence $E$ of objects of $\mathcal B$ so that
\begin{equation}E\Rightarrow \Phi(H^n(a^\bullet)) \ \ \text{ and } \  \ E_1^{pq}={\bf R}^q \Phi(a^p).\end{equation}
\end{Proposition}
\begin{proof}
Consider an injective Cartan-Eilenberg resolution $I^{\bullet \bullet}$ of $a^\bullet$ with an injection $a^\bullet \mono I^{\bullet,0}$ and differentials $d_{\rm I}^{pq}:I^{pq}\to I^{p+1,q}$ and $d_{\rm II}^{pq}:I^{pq}\to I^{p,q+1}.$ Then there are two spectral sequences ${}_{\rm I}E$ and $_{\rm II}E$ associated with the bicomplex $B^{\bullet\bullet}=\Phi(I^{\bullet\bullet})$ which converge to the cohomology of the totalisation $H^n({\sf Tot}(B^{\bullet\bullet}))$ so that ${}_{\rm I}E_1^{pq}=H^q_{\rm II}(B^{p\bullet}),$ ${}_{\rm I}E_2^{pq}=H^p_{\rm I}(H^q_{\rm II}(B^{\bullet\bullet}))$ and  ${}_{\rm II}E_1^{pq}=H^p_{\rm I}(B^{\bullet, q}),$ ${}_{\rm II}E_2^{pq}=H^q_{\rm II}(H^p_{\rm I}(B^{\bullet\bullet})).$ Since $I^{\bullet\bullet}$ is a Cartan-Eilenberg resolution, the complex $I^{\bullet,q}$ is homotopy equivalent to the complex of its cohomologies $H_{\rm I}^*(I^{\bullet,q})$ with zero differentials, and hence $H_{\rm I}^p(B^{\bullet,q })=\Phi(H^p_{\rm I}(I^{\bullet,q})).$  The complex $H^p_{\rm I}(I^{\bullet\bullet})$ is an injective resolution of $H^p(a^\bullet).$ It follows that $H^q_{\rm II}(H^p_{\rm I}(B^{\bullet\bullet}))={\bf R}^q\Phi(H^p(a^\bullet)).$ Using the fact that the objects $H^p(a^\bullet)$ are $\Phi$-acyclic, we obtain
\begin{equation}{}_{\rm II}E_2^{pq}=\left\{\begin{array}{l l}
0, & q\ne 0 \\
\Phi(H^p(a^\bullet)), & q=0.
\end{array}\right.\end{equation} Hence we get $H^n({\sf Tot}(B^{\bullet\bullet}))=\Phi(H^n(a^\bullet)).$ Therefore the spectral sequence $E:={}_{\rm I}E$ converges to $\Phi(H^n(a^\bullet)),$ and since $I^{p,\bullet}$ is an injective resolution of $a^p$, we have $E_1^{pq}=H_{\rm II}^q(\Phi(I^{p,\bullet}))={\bf R}^q\Phi (a^\bullet).$
\end{proof}
\begin{Corollary}\label{corollary_spectral} Let
$\mathcal F^\bullet$ be a bounded below complex of representations
of a category $\mathcal C$ with $\ilimit$-acyclic cohomologies.
Then there exists a cohomological  spectral sequence $E$ such that
\begin{equation}E \Rightarrow \ilimit\: H^n(\mathcal F^\bullet) \ \ \text{ and } \ \ E_1^{pq}=\ilimit^q \mathcal F^p.
\end{equation}
\end{Corollary}

\section{Representations of categories with pairwise coproducts.}\label{section_Representations_of_categories_with_pairwise_coproducts}

\subsection{Shifting of higher limits.}
 Let $\mathcal{C}$ be a category with pairwise coproducts i.e. for any objects  $c_1,c_2\in \mathcal{C}$ there exists the coproduct $c_1 \overset{i_1}{\longrightarrow} c_1\sqcup c_2 \overset{i_2}{\longleftarrow} c_2$ in $\mathcal{C}$ ($\mathcal{C}$ is not necessary has an initial object).
For any representation $\mathcal{F}$ of $\mathcal{C}$ there is a
natural transformation
\begin{equation}{\sf T}_{\mathcal F,c}:\mathcal{F}(c)\oplus \mathcal{F}(c) \longrightarrow \mathcal{F}(c\sqcup c),
\end{equation}
given by ${\sf T}_{\mathcal
F,c}=(\mathcal{F}(i_1),\mathcal{F}(i_2)).$ A representation
$\mathcal{F}$ of $\mathcal{C}$ is said to be {\it additive (resp.
monoadditive, epiadditive)} if ${\sf T}_{\mathcal F}$  is an
isomorphism (resp. monomorphism, epimorphism). If we put $\mathcal
F_{\sf sq}(c)=\mathcal F(c\sqcup c)$, we can write the previous
natural transformation as ${\sf T}_{\mathcal
F}:\mathcal{F}^2\longrightarrow \mathcal{F}_{\sf sq}.$ By a
representation $\mathcal F$ we construct the representation
$\Sigma\mathcal{F}:=\mathrm{coker}({\sf T}_{\mathcal F}),$ and by
recursion $\Sigma^{n+1} \mathcal{F}=\Sigma(\Sigma^n\mathcal{F})$.
Then a representation $\mathcal{F}$ is monoadditive if and only if
the following sequence is exact
\begin{equation}0\longrightarrow \mathcal{F}^2\overset{\sf T_{\mathcal F}}\longrightarrow \mathcal{F}_{\sf sq} \longrightarrow \Sigma\mathcal{F}\longrightarrow 0
\end{equation}
\begin{Proposition}\label{proposition_cokernel_limits}
Let $\mathcal{C}$ be a category with pairwise coproducts and $\mathcal{F}$
be a monoadditive representation of $\mathcal{C}.$ Then there is
an isomorphism
\begin{equation}\ilimit^n\:\mathcal{F}\cong \ilimit^{n-1}\: \Sigma\mathcal{F}
\end{equation}
for any $n\geq 0.$
\end{Proposition}
\begin{Corollary}
If $\mathcal{F}$ is a monoadditive representation of a category
with pairwise coproducts, then $\ilimit\,\mathcal{F}=0.$
\end{Corollary}
\begin{Corollary}\label{corollary_monoadditive}
If $\mathcal{F}$ is a representation of a category with pairwise coproducts such that the functor $\Sigma^n\mathcal{F}$
is monoadditive for $0\leq n<l$, then
\begin{equation}\ilimit^n\: \mathcal{F}=\ilimit^{n-l}\: \Sigma^l\mathcal{F}
\end{equation} for any $n\geq 0.$ In particular, $\ilimit^n \mathcal{F}=0$ for $0\leq n<l.$
\end{Corollary}
\begin{Corollary}
If $\mathcal{F}$ is an additive representation of a category with pairwise
coproducts, then $\ilimit^n\mathcal{F}=0$ for any $n\geq 0.$
\end{Corollary}
In order to prove  proposition \ref{proposition_cokernel_limits},
we need to prove two lemmas. The following lemma  seems to be well-known, but we provide a proof for completeness.

\begin{Lemma}\label{lemma_coproduct_contractible} A category $\mathcal{C}$ with pairwise coproducts is contractible.
\end{Lemma}
\begin{proof}
 Chose an object $c_0\in
\mathcal{C}$ and consider the functor $\Phi:\mathcal{C}\to
\mathcal{C}$ given by the formula $\Phi(-)=-\sqcup c_0.$ By
definition of coproduct, we obtain natural transformations from
the identity functor $\mathsf{Id}_\mathcal{C}\to \Phi$ and from
the constant functor $\mathsf{Const}_{c_0}\to \Phi.$ Using the
fact that a natural transformation of two functors induces  a
homotopy between the corresponding continuous maps on the
classifying spaces, we obtain that the identity map of the
classifying space $\mathsf{id}_{\mathsf{B}\mathcal{C}}$ is
homotopic to $\mathsf{B}\Phi:\mathsf{B}\mathcal{C}\to
\mathsf{B}\mathcal{C},$ and the map $\mathsf{B}\Phi$ is homotopic
to the constant map $\mathsf{const}_{c_0}$. Therefore,
$\mathsf{id}_{\mathsf{B}\mathcal{C}}\sim \mathsf{const}_{c_0},$
and hence, $\mathcal{C}$ is contractible.
\end{proof}
\begin{Lemma}\label{Lemma_coproduct_projections}
Let $\mathcal{F}$ be a representation of a category with coproducts. Then the morphisms $i_1,i_2:c\to c\sqcup c$ induce isomorphsms
\begin{equation}(i_k)_*:\ilimit^n\: \mathcal{F}(c) \overset{\cong}{\longrightarrow} \ilimit^n\:  \mathcal{F}(c\sqcup c).
\end{equation}
\end{Lemma}
\begin{proof}We denote by ${\sf sq}$ a functor $\mathcal C \to \mathcal C$ given by $c\mapsto c\sqcup c.$ Then $\mathcal F_{\sf sq}=\mathcal F\circ {\sf sq}.$
An object of the comma category $\mathsf{sq}\!\downarrow\! c_0$ is a pair $(c,\alpha:c\sqcup c\to c_0).$ An arrow $\alpha:c\sqcup c\to c_0$ is defined by the pair of arrows $\alpha_1=\alpha\circ i_1:c\to c_0$ and $\alpha_2=\alpha\circ i_2:c\to c_0.$ Thus the category $\mathsf{sq}\! \downarrow\! c_0$ is isomorphic to the category, whose objects are triples $(c,\alpha_1:c\to c_0,\alpha_2:c\to c_0)$ and morphisms $f:(c,\alpha_1,\alpha_2)\to (c',\alpha_1',\alpha_2')$ are morphisms $f:c\to c'$ such that $\alpha_1'f=\alpha_1$ and $\alpha_2'f=\alpha_2.$ It is easy to see that the object $(c\sqcup c',\alpha_1+\alpha_1',\alpha_2+\alpha_2')$ is the coproduct of objects $(c,\alpha_1,\alpha_2)$ and $(c',\alpha_1',\alpha_2')$ in this category. Hence the category  $\mathsf{sq}\!\downarrow\! c_0$ has coproducts and by lemma \ref{lemma_coproduct_contractible} it is contractible. Therefore, the functor ${\sf sq}$ satisfies the condition of proposition \ref{proposition_A}. Finally, applying corollary \ref{corollary_A_naturalness} to the natural transformation $i_k:{\sf Id}_{\mathcal{C}}\to {\sf sq},$ we obtain the claimed isomorphism.
\end{proof}

\begin{proof}[Proof of proposition \ref{proposition_cokernel_limits}]
Consider the long exact sequence of higher limits assotiated with
the short exact sequence $0\longrightarrow
\mathcal{F} \oplus \mathcal{F}
\longrightarrow \mathcal{F}_{\sf sq} \longrightarrow
\mathcal{F}'\longrightarrow 0.$ Using lemma
\ref{Lemma_coproduct_projections}, we obtain that the homomorphism
$((i_1)_*,(i_2)_*):\ilimit^n\:\mathcal{F}\oplus
\ilimit^n\:\mathcal{F}\longrightarrow
\ilimit^n\mathcal{F}_{\sf sq}$ is an epimorphism, and hence the
map $\ilimit^n\mathcal{F}_{\sf sq}\longrightarrow
\ilimit^n\:\mathcal{F}'$  vanishes. Therefore  we get the short exact sequences
\begin{equation}0\longrightarrow \ilimit^n\: \mathcal{F}'\longrightarrow \ilimit^{n+1}\:\mathcal{F}\oplus \ilimit^{n+1}\:\mathcal{F}\longrightarrow \ilimit^{n+1}\:\mathcal{F}_{\sf sq} \longrightarrow 0,
\end{equation} which are the totalisations of the bicartesian squares
$$\xymatrix{
\ilimit^n\:\mathcal{F}' \ar[rr]^{s_1}\ar[d]^{s_2} && \ilimit^{n+1}\:\mathcal{F}\ar[d]^{(i_1)_*}_{\cong} \\
\ilimit^{n+1}\:\mathcal{F}\ar[rr]^{(i_2)_*}_\cong &&
\ilimit^{n+1}\mathcal{F}_{\sf sq}.}$$ By lemma
\ref{Lemma_coproduct_projections} $(i_1)_*$ and $(i_2)_*$ are
isomorphisms, and hence $s_1$ and $s_2$ are isomorphisms.
\end{proof}

\subsection{Relative additivity and vanishing of higher limits.}
We denote by ${\sf Mod}$ the category of modules over  $k$-algebras. Its objects are pairs $(A,M)$ where $A$ is an algebra, $M$ is an $A$-module and a morphism $(A,M)\to (B,N)$ is a pair $(\varphi,f)$ where $\varphi:A\to B$ is a morphism of algebras and $f:M\to N$ is a linear map such that $f(am)=\varphi(a)f(m).$
If $M$ is an $A$-module, $N$ is a $B$-module and $\varphi:A\to B$ is a homomorphism of algebras, we will denote by $N\!\!\downarrow_\varphi$ the reduced $A$-module, and by $M\!\!\uparrow_\varphi=B\otimes_A M$ the induced $B$-module. If $f$ is obvious we denote $N\!\!\downarrow_A:=N\!\!\downarrow_\varphi$ and $M\!\!\uparrow^B:=M\!\!\uparrow_\varphi.$
Then for a morphism $(\varphi,f):(A,M)\to (B,N),$ the  map $f$ can be considered as an $A$-homomorphism $f:M\to N\!\!\downarrow_\varphi$. Since the functor $\uparrow_\varphi$ is left adjoint to  $\downarrow_\varphi$, a homomorphism $f:M\to N\!\!\downarrow_\varphi$ induces a homomorphism $f^\#:M\!\!\uparrow_\varphi\to N$ given by $f^\#(b\otimes m)=b\cdot f(m).$

Let $\mathcal{C}$ be a category with coproducts and $\mathcal O: \mathcal C\to {\sf Alg}$ be a functor to the category of $k$-algebras. {\it $\mathcal O$-representation} is a functor $\mathcal F:\mathcal C\to {\sf Mod}$ such that the diagram
$$\xymatrix{
 & & {\sf Mod}\ar[d] \\
\mathcal C\ar[rru]^{\mathcal F}\ar[rr]^{\mathcal O} & & {\sf Alg}
}$$
is commutative, where the right-hand functor is given by $(A,M)\mapsto A.$ Hence $\mathcal F(c)$ could be written as $(\mathcal O(c),\mathcal F_{\bf m}(c)),$ where $\mathcal F_{\bf m}(c)$ is a $\mathcal O(c)$-module. We will usually consider $\mathcal F(c)$ not as a pair but as the  vector space $\mathcal F_{\bf m}(c)$ with the fixed structure of $\mathcal O(c)$-module.

For an algebra $A$ we denote by $A^{\rm op}$ the opposite algebra with multiplication $a*b=ba.$ Similarly for a functor $\mathcal O:\mathcal C\to {\sf Alg}$ we denote by $\mathcal O^{\rm op}$ the functor given by $\mathcal O^{\rm op}(c)=\mathcal O(c)^{\rm op}.$ If $\mathcal F$ is an $\mathcal O$-representation and $\mathcal H$ is an $\mathcal O^{\rm op}$-representation, then one can define the representation $\mathcal H\otimes_{\mathcal O}\mathcal F:\mathcal C\to {\sf Mod}(k)$ given by $(\mathcal H\otimes_{\mathcal O}\mathcal F)(c)=\mathcal H(c)\otimes_{\mathcal O(c)}\mathcal F(c).$

Let $\mathcal C$ be a category with pairwise coproducts.  An $\mathcal O$-representation $\mathcal F$ is said to be additive if the morphism
\begin{equation}{\sf T}_{\mathcal F}:\mathcal F(c)\!\!\uparrow_{i_1}\oplus \ \mathcal F(c)\!\!\uparrow_{i_2}\longrightarrow \mathcal F(c\sqcup c)\end{equation}
is an isomorphism of $\mathcal O(c\sqcup c)$-modules, where ${\sf T}_{\mathcal F}=(\mathcal F(i_1)^\#,\mathcal F(i_2)^\#)$.

\begin{Proposition}\label{additive_limit_vanish} Let $\mathcal{C}$ be a category with pairwise coproducts, $\mathcal O:\mathcal C \to {\sf Alg}$ be a functor,  $\mathcal{F}$ be an additive $\mathcal O$-representation and $\mathcal H$ be an $\mathcal O^{\rm op}$-representation. Then
\begin{equation}\ilimit^i\ \mathcal H\otimes_{\mathcal O} \mathcal F =0\ \ \text{ for any } i\geq 0.\end{equation}
\end{Proposition}
\begin{proof}
We put ${\sf t}_{l,c}:=\mathcal H(i_l):\mathcal H(c)\to \mathcal H(c\sqcup c)\!\!\downarrow_{i_l}$ and $\sigma_l \mathcal H:={\rm coker}({\sf t}_l)$ for $l\in \{1,2\}.$ Since $({\sf id}_c,{\sf id}_c) \circ i_l={\sf id}_c,$ the morphism ${\sf t}_{l,c}$ is a $\mathcal O(c)$-split monomorphism, and hence $\mathcal H(c\sqcup c)\downarrow_{i_l}\cong \mathcal H(c)\oplus \sigma_l\mathcal H(c).$ Consider a commutative diagram
$$\xymatrix{
\bigoplus\limits_{l\in \{1,2\}} \mathcal H(c\sqcup c)\!\downarrow_{i_l} \otimes_{\mathcal O(c)} \mathcal F(c)  \ar[rr]_{\cong}^{{\sf id}\otimes {\sf T}_{\mathcal F}} && \mathcal H(c\sqcup c)\otimes_{\mathcal O(c\sqcup c)} \mathcal F(c\sqcup c)   \\
\bigoplus\limits_{l\in \{1,2\}} \mathcal H(c)\otimes_{\mathcal
O(c)} \mathcal F(c)\ar@(r,d)[rru]_{{\sf T}_{\mathcal
H\otimes_{\mathcal O} \mathcal F},c}\ar[u]^{ ({\sf t}_1\otimes
{\sf id}_{\mathcal F(c)})\oplus ({\sf t}_2\otimes {\sf
id}_{\mathcal F}(c)) } }$$ Since ${\sf t}_l$ is a $\mathcal
O(c)$-split monomorphism, the left-hand morphism is a split
monomorphism, and hence $\mathcal H\otimes_{\mathcal O}\mathcal F$
is monoadditive. Moreover, we obtain $\Sigma (\mathcal
H\otimes_{\mathcal O}\mathcal F) = \bigoplus_{l\in \{1,2\}}
\sigma_l\mathcal H\otimes_{\mathcal O} \mathcal F=(\sigma_1
\mathcal H\oplus \sigma_2\mathcal H)\otimes_{\mathcal O} \mathcal
F.$ Using the fact that  $\Sigma(\mathcal H\otimes_{\mathcal O}
\mathcal F)$ has the same form as $\mathcal H\otimes_{\mathcal
O}\mathcal F$ and  induction, we get that $\Sigma^n(\mathcal
H\otimes_{\mathcal O}\mathcal F)$ is monoadditive for all $n,$ and
thus, $\ilimit^n \ (\mathcal H\otimes_{\mathcal O}\mathcal F)=0.$

\end{proof}

\section{4-term exact sequences}\label{section_4-term_exact_sequences}
A natural way to get different presentations of homology groups as
derived limits is to use certain homological 4-term exact
sequences of the form
$$
\text{Homology}(\text{object})\hookrightarrow
\text{Functor}_1(\text{presentation})\to
\text{Functor}_2(\text{presentation})\twoheadrightarrow
\text{Homology}(\text{object})
$$
Generally such sequences appear from homological spectral
sequences. Here we give the examples of 4-term sequences and
related computations of $\ilimit, \ilimit^1$-functors in two
cases: 1) group homology, 2) cyclic homology of algebras.
\subsection{Relation modules and group homology}  For a group $G$, consider a
free presentation of $G$:
\begin{equation}\label{exacts1}
1\to R\to F\buildrel{\pi}\over\to G\to 1, \end{equation} where $F$
is a free group with a generating set $S$. The augmentation ideal
$\Delta(F)=\ker \{\mathbb Z[F]\to \mathbb Z\}$ of $F$ is the free
${\mathbb Z}[F]$-module on the set $\{ s-1 : s \in S \}$. There is
an exact sequence of ${\mathbb Z}G$-modules
\begin{equation}\label{magnus}
 0 \longrightarrow R_{ab}
   \stackrel{\mu}{\longrightarrow}
   {\mathbb Z}G \otimes _{{\mathbb Z}[F]} \Delta(F)
   \stackrel{\sigma}{\longrightarrow}
   \Delta(G)\longrightarrow 0 ,
\end{equation}
where $\mu$ maps $r[R,R]$ onto $1 \otimes (r-1)$ for all $r \in
R$, $\sigma$ maps the basis element $1 \otimes (s-1)$ onto
$\pi(s)-1$ for all $s \in S$ and $\Delta(G)$ is the augmentation
ideal in the group ring $\mathbb Z[G]$. The map $\mu$ is called
the {\it Magnus embedding} of the relation module $R_{ab}$.

Let $M$ be a $\mathbb Z[G]$-module. The module $M$ can be viewed
as a module over $R$ and $F$ where the action of $R$ is trivial
and the action of $F$ is induced by the action of $G$. Consider
the homology spectral sequence for the extension (\ref{exacts1})
with the coefficients $M$:
$$
E_{pq}^2=H_p(G, H_q(R, M))\Rightarrow H_{p+q}(F,M).
$$
Since all homology groups of $R$ and $F$ are zero in dimensions
$\geq 2$, this spectral sequence gives the natural isomorphisms:
\begin{equation}\label{isod2}
E_{p,0}^2=H_p(G,M)\buildrel{d_{p0}^2}\over\longrightarrow
E_{p-2,1}^2=H_{p-2}(G,R_{ab}\otimes M),\ p\geq 3
\end{equation}
and the following 4-term exact sequence:
\begin{equation}\label{4termhom}
0\to H_2(G, M)\to H_0(G, R_{ab}\otimes M)\to H_1(F,M)\to
H_1(G,M)\to 0.
\end{equation}
For $n\geq 1$, take the module $R_{ab}^{\otimes n}\otimes M$ as
coefficients in homology. The sequence (\ref{4termhom}) and
isomorphism (\ref{isod2}) induce the following 4-term exact
sequence:
\begin{equation}\label{4sq}
0\to H_{2n}(G,M)\to H_0(G,R_{ab}^{\otimes n}\otimes M)\to
H_1(F,R_{ab}^{\otimes n-1}\otimes M)\to H_{2n-1}(G,M)\to 0.
\end{equation}
Observe that, for $n\geq 2$,
\begin{equation}\label{triviality}
\ilimit\ H_1(F, R_{ab}^{\otimes n-1}\otimes M)=\ilimit^1 H_1(F,
R_{ab}^{\otimes n-1}\otimes M)=0.
\end{equation}
For the proof see the proof of theorem \ref{fibretheorem}. The
exact sequence (\ref{4sq}) and triviality of limits
(\ref{triviality}) implies that, for $n\geq 2$,
\begin{align*}
& \ilimit\ H_0(G, R_{ab}^{\otimes n}\otimes M)=H_{2n}(G,M)\\
& \ilimit^1 H_0(G, R_{ab}^{\otimes n}\otimes M)=H_{2n-1}(G,M).
\end{align*}
In the next section we will show how to extend this result to the
higher limits.

Now we will consider the similar 4-term sequence, where the
homology group is presented not as a kernel, but as a cokernel.
\begin{Proposition}
The sequence (\ref{magnus}) induces the following 4-term exact
sequence
$$
0\to H_1(F, R_{ab}\otimes M)\to H_1(F,{\mathbb Z}[G] \otimes
_{{\mathbb Z}[F]} \Delta(F)\otimes M)\to H_1(F,\Delta(G)\otimes
M)\to H_2(G,M)\to 0
$$
\end{Proposition}
\begin{proof}
Applying the functor $H_*(F,-)$ to the Magnus embedding, we get
three left terms of the needed sequence. It remains to show that
the image of the boundary map
$$
H_1(F, \Delta(G)\otimes M)\to H_0(F, R_{ab}\otimes M)
$$
is exactly $H_2(G,M)$. This follows from the natural commutative
diagram
$$
\xyma{H_1(F,\Delta(G)\otimes M)\ar@{->>}[d] \ar@{->}[r] &
H_0(F,R_{ab}\otimes M)\ar@{=}[d] \\ H_1(G,\Delta(G)\otimes
M)\ar@{->}[r] & H_0(G,R_{ab}\otimes M)}
$$
and the isomorphism $H_1(G,\Delta(G)\otimes M)=H_2(G,M)$.
\end{proof}
\begin{Corollary}
For $n\geq 1$, there is the following exact 4-term sequence
\begin{multline}\label{4multline} 0\to H_1(F, R_{ab}^{\otimes n}\otimes M)\to
H_1(F,R_{ab}^{\otimes n-1}\otimes {\mathbb Z}[G] \otimes
_{{\mathbb Z}[F]}\Delta(F)\otimes M)\to\\ H_1(F,R_{ab}^{\otimes
n-1}\otimes \Delta(G)\otimes M)\to H_{2n}(G,M)\to 0
\end{multline}
\end{Corollary}
\subsection{Quillen sequences}
Let $k$ be a field of characteristic zero, $E$ an associative
$k$-algebra, $I$ a two-sided ideal of $E$ and $M$ an $E$-bimodule.
Let $\mathcal B(E)$ be the free bimodule resolution of  $E$:
$\mathcal B(E)_n=E^{\otimes n+2},$ $(n+2)$-fold tensor product
over $k$, $n\geq 0$, with differential $b'$ given by
$$b'(x_0\otimes\ldots \otimes
x_{n+1})=\sum_{j=0}^n(-1)^j(x_0\otimes \ldots \otimes x_jx_{j+1}
\otimes \ldots \otimes x_{n+1}), x_i\in E.$$ Following \cite{Q},
denote by $M\otimes_E$ the Hochschild homology group $H_0(E,M)$;
set
$$ M\otimes_E^{!}=M\otimes_E\mathcal B(E)\otimes_E,\
(M\otimes_E^{!})^n=M\otimes_E\mathcal B(E)\otimes_E\dots
M\otimes_E\mathcal B(E) \otimes_E\ (n\ \text{times})$$ and denote
by $(M\otimes_E^{!})_{\mathbb Z/n}^n$ its quotient under the the
action the cyclic group $\mathbb Z/n$.

Recall from \cite{Q} (Theorem 5.5) that there exists the following
spectral sequence
$$E_{pq}^1=H_{q-p}((E\otimes^{!}_E)^{p+1}_{\mathbb Z/{p+1}}/
(I\otimes^{!}_E)^{p+1}_{\mathbb Z/{p+1}})\Rightarrow
HC_{p+q}(E/I).$$ In the case $E$ is free, the spectral sequence
collapses to the exact sequence (for $n\geq 1$)
$$
\xyma{HC_{2n}(E/I)\ar@{>->}[r] & E_{nn}^1 \ar@{->}[r]^{d_{nn}^1}
\ar@{=}[d] & E_{n-1,n}^1\ar@{->>}[r] \ar@{=}[d] & HC_{2n-1}(E/I)\\
& HC_0(E/I^{n+1}) & H_1(E,E/I^n)_\sigma}
$$
Since, for $n\geq 1$, $HC_0(E/I^{n+1})=E/(I^{n+1}+[E,E])$, where
$[E,E]$ is the $k$-submodule of $R$ generated by $rs-sr,\ r,s\in
E$, one has the following 4-term sequence:
\begin{equation}\label{Quillen_sequense}
0\to HC_{2n}(E/I)\to E/(I^{n+1}+[E,E])\to H_1(E,E/I^n)_\sigma\to
HC_{2n-1}(E/I)\to 0
\end{equation}
The similar spectral sequence argument is used in \cite{Q}
(Theorem 5.11) for an analogous 4-term spectral sequence for the
odd dimensional reduced cyclic homology (for $n\geq 0$):
\begin{equation}\label{Quillen_sequense2}
0\to \overline{HC}_{2n+1}(E/I)\to I^{n+1}/[I,I^n]\to H_1(E, I^n)\to \overline{
HC}_{2n}(E/I)\to 0
\end{equation}

Note that, by definition, $ H_1((E/I^n)\otimes^{!}_E)=H_1(E,
E/I^n)$.

\section{Group homology as higher limits}\label{section_Group_homology_as_higher_limits}

\subsection{Group homology}

In this section, we will consider the category ${\sf Pres}(G)$
of free presentations of a group $G$. The objects of ${\sf Pres}(G)$ are surjective homomorphisms $\pi:F\epi G$ where $F$ is
a free group and  morphisms $f:(\pi_1:F_1\epi G)\to (\pi_2:F_2\epi
G)$ are homomorphisms $f:F_1\to F_2$ such that $\pi_1=\pi_2f$. The
category ${\sf Pres}(G)$ has coproducts given by
$$(\pi_1:F_1\epi G)\sqcup(\pi_2:F_2\epi G)=(\pi_1+\pi_2:F_1*F_2\epi
G),$$ and hence it is contractible. In particular, for any abelian
group $A$ higher limits $\ilimit^i A$ vanish for $i>0.$
 We will always denote by $R$ the kernel of an epimorphism $\pi:F\epi G.$ Therefore an object of ${\sf Pres}(G)$ defines a short exact sequence of groups
$$1\longrightarrow R \longrightarrow F \overset{\pi}{\longrightarrow} G \longrightarrow 1,$$
and $R$ can be considered as a functor $R:{\sf Pres}(G)\to {\sf Gr}.$
All the limits considered bellow in this section are taken over the
category ${\sf Pres}(G)$.

\begin{Theorem}\label{groupisomorphism}
For a group $G$ and a $\mathbb Z[G]$-module $M$, for $n\geq 1$,
there are natural isomorphisms
\begin{equation}\label{mainformula}
\ilimit^i\:  H_0(G,R_{ab}^{\otimes n}\otimes M)=\begin{cases}
H_{2n-i}(G,M)\ \text{ for } \ i<n,\\ 0\ \text{for}\ i>n
\end{cases}
\end{equation}
For $i=n$, there is a natural short exact sequence
$$
0\to H_n(G,M)\to \ilimit^nH_0(G,R_{ab}^{\otimes n}\otimes M)\to
(\Delta(G)^{\otimes n-1}\otimes M)\Delta(G)\to 0
$$
\end{Theorem}
\begin{Corollary}\label{homology_corollary}
For any group $G$ and $n\geq 1$,
there are natural isomorphisms
$$
H_{2n-i}(G)=\ilimit^i\:  (R_{ab}^{\otimes n})_G \ \text{ for } \
i<n.
$$
\end{Corollary}
\begin{proof}
First lets prove the theorem for $n=1$. Observe that the
representation $H_1(F,M)$ is monoadditive and
$$
\Sigma H_1(F,M)=M\Delta(G).
$$
Hence, by proposition \ref{proposition_cokernel_limits},
$$
\ilimit^i H_1(F,M)=\begin{cases} M\Delta(G)\ \text{for}\ i=1\\
0,\ i\neq 1\end{cases}
$$
The needed isomorphisms follow now from the sequence
(\ref{4termhom}) and corollary \ref{corollary_spectral} applied to
the complex
$$
H_0(G, R_{ab}\otimes M)\to H_1(F,M)
$$
whose zeroth homology is $H_1(G,M)$ and the first homology is
$H_2(G,M)$. For $i=n=1,$ the spectral sequence
\ref{corollary_spectral} implies that there is following natural
short exact sequence
$$
0\to H_1(G,M)\to \ilimit^1H_0(G, R_{ab}\otimes M)\to M\Delta(G)\to
0.
$$
In particular, for a trivial $\mathbb Z[G]$-module $M$, there is
an isomorphism
$$
H_1(G,M)=\ilimit^1H_0(G,R_{ab}\otimes M).
$$

Now assume that the isomorphisms (\ref{mainformula}) are proved
for a given $n$. Consider the next step. Since the category
${\sf Pres}(G)$ is acyclic, the 4-term sequence (\ref{4sq})
implies that there is a natural isomorphism
$$
\ilimit^i H_0(G,R_{ab}^{\otimes n+1}\otimes M)= \ilimit^i
H_1(F,R_{ab}^{\otimes n}\otimes M),\ i\geq 2.
$$
For any $\mathbb Z[G]$-module $N$ (which may depend on $F$ and
$R$),
$$
\ilimit^i H_1(F, {\mathbb Z}[G] \otimes _{{\mathbb
Z}[F]}\Delta(F)\otimes N)=0,\ i\geq 0
$$
Now the sequence (\ref{4multline}) implies that
$$
\ilimit^i H_1(F, R_{ab}^{\otimes n}\otimes M)= \ilimit^{i-1}
W_n(F,R),\ n\geq 1
$$
for all $i\geq 0$, where $W_n(F,R)$ is the part of the following
short exact sequence
$$
0\to W_n(F,R)\to H_1(F, R_{ab}^{\otimes n-1}\otimes
\Delta(G)\otimes M)\to H_{2n}(G,M)\to 0.
$$
First consider the case $i=2$. For $n\geq 2$,
$$
\ilimit^i H_1(F,R_{ab}^{\otimes n-1}\otimes \Delta(G)\otimes
M)=0,\ i=0,1.
$$
Therefore, there is a natural isomorphism
$$
H_{2n}(G,M)=\ilimit^1\ W_n(F,R)=\ilimit^2 H_1(F,R_{ab}^{\otimes
n}\otimes M)=\ilimit^2\ H_0(F,R_{ab}^{\otimes n+1}\otimes M)
$$
and (\ref{mainformula}) follows for $i=2$. In the case $n=i=2$, we
get the natural exact sequence
$$
0\to H_2(G,M)\to \ilimit^2 H_0(F,R_{ab}^{\otimes 2}\otimes M)\to
(\Delta(G)\otimes M)\Delta(G)\to 0,
$$
where the last term comes from the isomorphism $$\ilimit^1
H_1(F,\Delta(G)\otimes M)=(\Delta(G)\otimes M)\Delta(G).$$

For $i\geq 3$, there is a natural isomorphism
$$
\ilimit^i H_1(F, R_{ab}^{\otimes n}\otimes M)\simeq \ilimit^{i-1}
H_1(F,R_{ab}^{\otimes n-1}\otimes \Delta(G)\otimes M).
$$
Since $$\ilimit^{i-1}H_1(F,R_{ab}^{\otimes n-1}\otimes
\Delta(G)\otimes M)= \ilimit^{i-1}H_0(G, R_{ab}^{\otimes n}\otimes
\Delta(G)\otimes M),$$ the inductive assumption implies that
$$
\ilimit^i H_1(F,R_{ab}^{\otimes n}\otimes M)=\begin{cases}
H_{2n-i+1}(G,\Delta(G)\otimes M)= H_{2n-i+2}(G,M)\ \text{for}\ i<n+1\\
0\ \text{for}\ i>n+1
\end{cases}
$$
For $i=n+1$, we get the natural exact sequence
$$
0\to H_{n+1}(G,M)\to \ilimit^{n+1}H_0(G,R_{ab}^{\otimes
n+1}\otimes M)\to (\Delta(G)^{\otimes n}\otimes M)\Delta(G)\to 0
$$
and the inductive step is completed for $i\geq 3$.
\end{proof}

\subsection{Fibre products} For $n\geq 2$ and an epimorphism $p:
F\to G$, consider the fibre product
$$
F_n(\pi):=\underbrace{F\times_G\times\dots\times_GF}_n=\{(x_1,\dots,
x_n)\in F^{\times n}\ |\ \pi(x_1)=\dots=\pi(x_n)\}
$$
There is a natural split exact sequence $(n\geq 2)$:
\begin{equation}\label{fibresplit}
1\to R\to F_n(\pi)\to F_{n-1}(\pi)\to 1.
\end{equation}
For any $\mathbb Z[G]$-module $M$, there is a split exact sequence
of homology groups which follows from the spectral sequence
applied to the extension (\ref{fibresplit}):
\begin{equation}\label{hnfp} 0\to H_i(F_{n-1}(\pi),R_{ab}\otimes M)\to
H_{i+1}(F_n(\pi), M)\to H_{i+1}(F_{n-1}(\pi),M)\to 0.
\end{equation}
For $i=n-1$, we have a natural isomorphism
\begin{equation}\label{isohp}
H_n(F_n(\pi),M)\simeq H_1(F, R_{ab}^{\otimes n-1}\otimes M).
\end{equation}
\begin{Theorem}\label{fibretheorem}
For any $\mathbb Z[G]$-module $M$, there are natural isomorphisms
\begin{align}
& \ilimit\ H_{n-1}(F_n(\pi),M)= H_{2n-2}(G,M)\ \text{for}\ n\geq 2\\
& \ilimit^1 H_{n-1}(F_n(\pi),M)= H_{2n-3}(G,M)\ \text{for}\ n\geq
3.
\end{align}
\end{Theorem}
\begin{proof}
We first show that
\begin{equation}\label{simplei}
\ilimit\ H_1(F_2(\pi), M)=H_2(G,M).
\end{equation}
To see this, consider the short exact sequence
$$
0\to H_0(F,R_{ab}\otimes M)\to H_1(F_2(\pi),M)\to H_1(F,M)\to 0
$$
The representation $H_1(F,M)$ is monoadditive, hence $\ilimit\
H_1(F,M)=0$. Now the isomorphism (\ref{simplei}) follows from
theorem \ref{groupisomorphism} for $n=1, i=0$.

Now we can assume that $n\geq 3$. We will show that, for any
module $N$ (which may depend on $F$ and $R$) and $k\geq 1$,
\begin{equation}\label{limlim1}
\ilimit\ H_1(F, R_{ab}^{\otimes k}\otimes N)=\ilimit^1
H_1(F,R_{ab}^{\otimes k}\otimes N)=0.
\end{equation}
For a presentation $F/R=G$, consider its free square:
$$
1\to \widetilde{R}\to F*F\to G\to 1,
$$
where $\widetilde{R}$ is the kernel of the natural projection
$F*F\twoheadrightarrow G,$ which sends both copies of $F$ via $p$.
The representation $H_1(F,R_{ab}^{\otimes k}\otimes N)$ is
monoadditive. The cokernel $Cok(F,R)$ of the natural map
$$
H_1(F, R_{ab}^{\otimes k}\otimes N)^{\oplus 2}\hookrightarrow
H_1(F*F, {\widetilde R}_{ab}^{\otimes n}\otimes N)
$$
can be described with the help of Mayer-Vietoris sequence for
group homology of a free product as follows. There is a natural
diagram
$$
\xyma{H_1(F,R_{ab}^{\otimes n}\otimes N)^{\oplus 2}\ar@{=}[r]
\ar@{>->}[d] & H_1(F,R_{ab}^{\otimes n}\otimes N)^{\oplus
2}\ar@{>->}[d]\\ H_1(F*F, R_{ab}^{\otimes n}\otimes N)\ar@{->>}[d]
\ar@{>->}[r] & H_1(F*F, \widetilde{R}_{ab}^{\otimes n}\otimes N)
\ar@{->}[r] \ar@{->>}[d] & H_1\left(F*F,
\frac{\widetilde{R}_{ab}^{\otimes
n}\otimes N}{R_{ab}^{\otimes n}\otimes N}\right)\\
\Delta(G)(R_{ab}^{\otimes n}\otimes N)\ar@{>->}[r]^{inj} &
Cok(F,R) \ar@{->>}[r] & Coker(inj)\ar@{>->}[u]}
$$
The triviality of $\ilimit^1$ in (\ref{limlim1}) now follows from
proposition \ref{proposition_cokernel_limits}, since both
representations $R_{ab}^{\otimes n}\otimes N$ and $H_1\left(F*F,
\frac{\widetilde{R}_{ab}^{\otimes n}\otimes N}{R_{ab}^{\otimes
n}\otimes N}\right)$ are monoadditive.

The exact sequence (\ref{hnfp}) for $i=n-2$, and the isomorphism
(\ref{isohp}) imply that
$$
\ilimit^j\ H_{n-1}(F_n(\pi), M)=\ilimit^j\ H_{n-2}(F_{n-1}(\pi),
R_{ab}\otimes M),\ \text{for}\ j=0,1.
$$
Continuing this process, we get the isomorphisms
$$
\ilimit^jH_{n-1}(F_n(\pi),M)=\ilimit^j H_0(F,R_{ab}^{n-1}\otimes
M),\ \text{for}\ j=0,1.
$$
The needed statement now follows from theorem
\ref{groupisomorphism}.
\end{proof}

\section{Hihger relation modules}\label{section_Hihger_relation_modules} For a free presentation of a
group $G$,
\begin{equation}\label{extension100}
1\to R\to F\to G\to 1, \end{equation}
the lower central series
quotients $\gamma_n(R)/\gamma_{n+1}(R)$ are called {\it higher
relation modules}. Here $\{\gamma_n(R)\}_{n\geq 1}$ is the lower
central series of $R$ defined inductively as $\gamma_1(R)=R,\
\gamma_{n+1}(R)=[\gamma_n(R),R]$. The action of $G$ on
$\gamma_n(R)/\gamma_{n+1}(R)$ is defined via conjugation in $F$.
The properties of functors
\begin{equation}\label{newm1}
\ilimit\ \gamma_n(R)/[\gamma_n(R),F]=\ilimit\ H_0(G,
\gamma_n(R)/\gamma_{n+1}(R))
\end{equation}
are considered in \cite{EM}. These functors are interesting since
the map between higher relation modules and tensor powers of the
ordinal relation module
$$
\xyma{ (\text{the}\ n\text{th Lie power of}\ R_{ab}=)&
\gamma_n(R)/\gamma_{n+1}(R)\ar@/_25pt/[rr]^{n} \ar@{->}[r] &
R_{ab}^{\otimes n}\ar@{->}[r] & \gamma_n(R)/\gamma_{n+1}(R)}
$$
\vspace{.5cm}

\noindent induces, for all $0\leq i<n$, the natural map between
corresponding limits
$$
\ilimit^i\ \gamma_n(R)/[\gamma_{n+1}(R),F]\to H_{2n-i}(G)\to
\ilimit^i\ \gamma_n(R)/[\gamma_{n+1}(R),F]
$$
by corollary \ref{homology_corollary}. Here we describe the
functor (\ref{newm1}) for $n=2$.

\begin{prop}\label{higherrel}
There is a natural isomorphism
\begin{equation}\label{g2isom}
\ilimit\ \gamma_2(R)/[\gamma_2(R),F]=H_2(G,S^2(\Delta(G))),
\end{equation}
where $S^2$ is the symmetric square functor. If $G$ is a
2-torsion-free group, then
\begin{align}
& \ilimit\ \gamma_2(R)/[\gamma_2(R),F]=H_4(G,\mathbb Z/2)\label{first1}\\
& \ilimit^1\ \gamma_2(R)/[\gamma_2(R),F]=H_3(G,\mathbb
Z/2).\label{second1}
\end{align}
\end{prop}
\begin{proof}
Recall that, for a short exact sequence of free abelian groups
\begin{equation}\label{ses} 0\to A\to B\to C\to 0,
\end{equation}
there exists a natural long exact sequence, called the Koszul
complex
$$
0\to \Lambda^2(A)\to A\otimes B\to S^2(B)\to S^2(C)\to 0,
$$
where $\Lambda^2$ is the exterior square functor. Using the Magnus
embedding instead of (\ref{ses}), we obtain the following exact
sequence, which is a sequence not only of abelian groups, but also
of $\mathbb Z[G]$-modules:
\begin{equation}\label{l2seq}
0\to \Lambda^2(R_{ab})\to R_{ab}\otimes (\mathbb Z[G]_{\mathbb
Z[F]}\Delta(F)) \to S^2(\mathbb Z[G]\otimes_{\mathbb
Z[F]}\Delta(F))\to S^2(\Delta(G))\to 0
\end{equation}
Denote
$$
N(R,F):=R_{ab}\otimes (\mathbb Z[G]\otimes_{\mathbb
Z[F]}\Delta(F))
$$
and the cokernel of the left hand map in the above sequence
$$
Cok(F,R):=coker\{\Lambda^2(R_{ab})\hookrightarrow N(F,R)\}
$$
After applying the functor $H_*(G,-)$ to the exact sequence of
$\mathbb Z[G]$-modules
$$
0\to \Lambda^2(R_{ab})\to N(F,R)\to Cok(F,R)\to 0
$$
we obtain the sequence
\begin{multline}\label{multseq}
0\to H_1(G, Cok(F,R))\to \gamma_2(R)/[\gamma_2(R),F]\to\\
H_0(G,N(F,R))\to H_0(G,Cok(F,R))\to 0.
\end{multline}
Now we observe that $H_0(F, N(F,R))$ is a monoadditive
representation and the sequence (\ref{multseq}) implies that there
is an isomorphism \begin{equation}\label{ilimg2} \ilimit\
H_1(G,Cok(F,R))=\ilimit\ \gamma_2(R)/[\gamma_2(R),F].
\end{equation}
Recall that, for any group $G$ and a free $\mathbb Z[G]$-module
$P$, $H_{2k}(G,S^2(P))=0,\ k\geq 1$ and $H_i(G,S^2(P))=0$ for all
$i$ if $G$ is 2-torsion-free. The description of $Cok(F,R)$ as a
kernel
\begin{equation}\label{coks2}
0\to Cok(F,R)\to S^2(\mathbb Z[G]\otimes_{\mathbb
Z[F]}\Delta(F))\to S^2(\Delta(G))\to 0
\end{equation}
implies that there is the following exact sequence
\begin{equation}\label{s2cok}
0\to H_2(G,S^2(G))\to H_1(G,Cok(F,R))\to H_1(G, S^2(\mathbb
Z[G]\otimes_{\mathbb Z[F]}\Delta(F)))
\end{equation}
Again we observe that $H_1(G, S^2(\mathbb Z[G]\otimes_{\mathbb
Z[F]}\Delta(F)))$ is a monoadditive representation. The sequence
(\ref{s2cok}) implies that
$$
\ilimit\ H_1(G,Cok(F,R))=\ilimit\ H_2(G,S^2(G))=H_2(G,S^2(G))
$$
and (\ref{g2isom}) follows from (\ref{ilimg2}).

Now assume that $G$ is 2-torsion-free. Then
$H_1(G,Cok(F,R))=H_2(G, S^2(G))=H_4(G,\mathbb Z/2)$ (see, for
example, \cite{KKS}) and the isomorphism (\ref{first1}) follows.
To compute $\ilimit^1$-term, consider the sequence (\ref{multseq})
and observe that $\ilimit^iH_1(G,Cok(F,R))=0,\ i\geq 1$, hence
$$
\ilimit^1\ \gamma_2(R)/[\gamma_2(R),F]=\ilimit\ H_0(G,Cok(F,R)).
$$
Applying the homology functor $H_*(G,-)$ to the sequence
(\ref{coks2}) we get the exact sequence \begin{multline*} 0\to
H_1(G, S^2(\Delta(G)))\to H_0(G, Cok(F,R))\to\\ H_0(G, S^2(\mathbb
Z[G]\otimes_{\mathbb Z[F]}\Delta(F)))\to H_0(G, S^2(\Delta(G)))\to
0.
\end{multline*}
The representation $H_0(G, S^2(\mathbb Z[G]\otimes_{\mathbb
Z[F]}\Delta(F)))$ is monoadditive, hence
$$
\ilimit\ H_0(G, Cok(F,R))=\ilimit\
H_1(G,S^2(\Delta(G)))=H_1(G,S^2(\Delta(G)))=H_3(G,\mathbb Z/2)
$$
and the isomorphism (\ref{second1}) follows.
\end{proof}

One can use the results and methods from \cite{KKS} and
\cite{Stohr} to get the descriptions of limits like (\ref{first1})
and (\ref{second1}). For example, consider the next higher
relation module. In order to compute the limit
$$
\ilimit\ \gamma_3(R)/[\gamma_3(R),F]
$$
consider the following exact sequence from \cite{Stohr}:
$$
0\to H_1(G, K_1^3)\to \gamma_3(R)/[\gamma_3(R),F]\to H_0(G,
\mathbb Z[G]\otimes_{\mathbb Z[F]}\Delta(F)\otimes S^2(R_{ab}))\to
H_0(G, K_1^3)\to 0,
$$
where $K_1^3$ is a $\mathbb Z[G]$-module which lives in the
following diagram with short exact sequences
$$
\xyma{& & R_{ab}\otimes S^2(\Delta(G))\ar@{>->}[d]\\ K_1^3
\ar@{>->}[r] & S^3(\mathbb Z[G]\otimes_{\mathbb Z[F]}\Delta(F))
\ar@{->>}[r] & \frac{S^3(\mathbb Z[G]\otimes_{\mathbb
Z[F]}\Delta(F))}{K_1^3}\ar@{->>}[d]\\ & & S^3(\Delta(G))}
$$
Suppose that $G$ is a 3-torsion-free group. Then the torsion
subgroup of the quotient $\gamma_3(R)/[\gamma_3(R),F]$ consists of
3-torsion elements only \cite{Stohr}. The reasoning from the
proof of proposition \ref{higherrel} leads to the following:\\ \\
\noindent {\it for a 3-torsion-free group $G$, $ \ilimit\
\gamma_3(R)/[\gamma_3(R),F]=H_4(G,\mathbb Z/3). $}
\section{Hochschild and Cyclic homology.}\label{section_Hochschild_and_Cyclic_homology}

Let $k$ be a field. All algebras in this section are assumed to be associative unital algebras over $k,$ and  $\otimes=\otimes_k.$  If $F$ is an algebra and $M$ is an $F$-bimodule, we denote by $F^e=F\otimes F^{op}$ the enveloping algebra of $F$ and following Quillen we denote the space of  coinvariants   $M_\natural:=M/[F,M]=H_0(F,M).$

\subsection{Magnus embedding for algebras.}

Let  $A_1$ and $A_2$ be algebras and $L$ and $N$ be $(A_1,A_2)$-bimodules. A {\it  bimodule derivation} $\delta: L\to N$ is a $k$-linear map such that
$\delta(amb)=\delta(am)b-a\delta(m)b+a\delta(mb).$ Given a left-right split short exact sequence of $(A_1, A_2)$-bimodules
\begin{equation}\label{right-left_split}
0 \longrightarrow N \overset{f}{\longrightarrow} M \overset{g}{\longrightarrow} L \longrightarrow 0
\end{equation}
 one associates to it a bimodule derivation by choosing to it for $f$ a left $A_1$-splitting $l:L\to M$ and a right $A_2$-splitting $r:L\to M.$  There is then a unique $k$-linear map $\delta : L\to N$ such that $g\delta=l-r$ and this map is a bimodule derivation.

Let $B_1=A_1/I_1$ and $B_2=A_2/I_2$ be quotient algebras and $M$ be a $(A_1,A_2)$-bimodule. Then we put $M\!\!\uparrow^{(B_1,B_2)}:=M/(I_1M+MI_2)\cong B_1\otimes_{A_1} M\otimes_{A_2} B_2.$ It gives a well defined functor $(-)\!\!\uparrow^{(B_1,B_2)}: {\sf Bimod}(A_1,A_2)\to {\sf Bimod}(B_1,B_2).$ For the sake of simplicity, we put $\uparrow:=\uparrow^{(B_1,B_2)}.$

\begin{Lemma}{\rm (\cite[proposition 3.2]{Butler_King}\rm )}
As above, let \eqref{right-left_split} be a right-left split sequence of $(A_1,A_2)$-bimodules, let $\delta: L\to N$ be an associated bimodule derivation and $B_i=A_i/I_i$ be quotient algebras. Then there is an exact sequence of $(B_1,B_2)$-bimodules
$$\frac{I_1M\cap MI_2}{I_1MI_2}\longrightarrow \frac{I_1L\cap LI_2}{I_1LI_2}\overset{\tilde\delta }\longrightarrow N\!\!\uparrow \overset{f\uparrow}\longrightarrow M\!\!\uparrow \overset{g\uparrow}\longrightarrow L\!\!\uparrow \longrightarrow 0,$$ where $\tilde \delta$ is the bimodule homomorphism induced by $\delta.$
\end{Lemma}

We denote by $\Omega(A)$ the universal derivation $A$-bimodule of a $k$-algebra $A$. It is defined as the kernel of the multiplication map $\Omega(A)={\rm ker}(\mu : A\otimes A \to A).$ The universal derivation $\delta:A \to \Omega(A)$ is given by $\delta(a)=a\otimes 1-1\otimes a,$ and it is the associated bimodule derivation to the left-right split short exact sequence
\begin{equation}
0 \longrightarrow \Omega(A)\longrightarrow A\otimes A \overset{\mu }\longrightarrow A \longrightarrow 0
\end{equation}
with splittings $l(a)=a\otimes 1$ and $r(a)=1\otimes a.$ Let $\bar A$ denote the vector space quotient $A/k.$ Then there is a well-defined isomorphism of vector spaces  $ A\otimes \bar A\cong \Omega(A)$ given by $x\otimes \bar y\mapsto x\otimes y-xy\otimes 1$ (see \cite[proposition 10.1.3]{Ginzburg}). The composition of $\delta$ with this isomorphism is the map $A\to A\otimes \bar A$ given by $a\mapsto 1\otimes \bar a.$ It follows that
\begin{equation}\label{ker_coker_delta}
{\rm ker}(\delta)=k  \ \ \ \text{ and } \ \ \ {\rm coker}(\delta)=\bar A^{\otimes 2}.
\end{equation}

Let $B=A/I$ be a quotient algebra, then we denote the enveloping algebra by $B^e=B\otimes B^{\rm op}$ and $\uparrow^{B^e}=\uparrow^{(B,B)}:{\sf Bimod}(A)\to {\sf Bimod}(B).$

\begin{Proposition}
Let $A$ be a $k$-algebra and $B=A/I$ be its quotient algebra. Then there is an exact sequence of $B$-bimodules
$$0\longrightarrow I/I^2 \overset{\tilde \delta}\longrightarrow \Omega(A)\!\!\uparrow^{B^e} \overset{d}\longrightarrow B\otimes B \overset{\mu}\longrightarrow B \longrightarrow 0,$$ where $\mu$ is the multiplication map, $d$ is induced by inclusion $\Omega(A)\hookrightarrow A\otimes A$ and $\tilde\delta$ is induced by the universal derivation $\delta:A\to \Omega(A).$
\end{Proposition}
\begin{proof}
It follows from the previous lemma and the equations $B= A\!\!\uparrow^{B^e},$ $B\otimes B= (A\otimes A)\!\!\uparrow^{B^e}$, $I/I^2=(IA\cap AI)/(IAI)$ and $((I\otimes A)\cap (A\otimes I))/(I\otimes I)=0.$
\end{proof}

If $F$ is a (quasi-)free algebra and $A\cong F/I,$ then $A$-bimodule $I/I^2$ is said to be {\it relation bimodule}. By the previous proposition we have a short exact sequence of $A$-bimodules
\begin{equation}\label{Magnus_s.e.s}
0 \longrightarrow I/I^2\overset{\tilde \delta}\longrightarrow \Omega(F)\!\!\uparrow^{A^e} \longrightarrow \Omega(A)\longrightarrow 0.
\end{equation}
$\Omega(F)$ is a projective $F$-bimodule, and hence, $ \Omega(F)\!\!\uparrow^{A^e}$ is a projective $A$-bimodule. Since $\Omega(A)$ is left-right projective, the sequence \eqref{Magnus_s.e.s} left-right splits, and hence $I/I^2$ is left-right projective. The map $\tilde\delta: I/I^2\to \Omega(F)\!\!\uparrow^{A^e}$ is called Magnus embedding of the relation bimodule $I/I^2.$

\subsection{Higher limit formula for Hochschild homology.}

In this section, we will consider the category ${\sf Pres}(A)$
of free presentations of an algebra $A$. The objects of ${\sf Pres}(A)$ are surjective homomorphisms $\pi:F\epi A$ where $F$ is a free algebra and morphisms $f:(\pi_1:F_1\epi A)\to (\pi_2:F_2\epi A)$ are homomorphisms $f:F_1\to F_2$ such that $\pi_1=\pi_2f$. The category ${\sf Pres}(G)$ has coproducts given by $(\pi_1:F_1\epi G)\sqcup(\pi_2:F_2\epi G)=(\pi_1+\pi_2:F_1*F_2\epi G),$ where $F_1*F_2$ is the free product of algebras. Therefore this category is contractible. In particular, for any vector space $V$ higher limits $\ilimit^i V$ vanish for $i>0.$
 We will always denote by $I$ the kernel of an epimorphism $\pi:F\epi A.$ Hence an object of ${\sf Pres}(G)$ defines a short exact sequence
\begin{equation}1\longrightarrow I \longrightarrow F \overset{\pi}{\longrightarrow} A \longrightarrow 1.
\end{equation}
and $I$ can be considered as a functor $I:{\sf Pres}(A)\to {\sf Mod}(A).$
All the limits considered bellow in this section are taken over the
category ${\sf Pres}(A)$.

\begin{Lemma}\label{omega_vanish_hiher_limits} If we consider $F^e$ as a functor $F^e:{\sf Pres}(A)\to {\sf Alg},$ sending $(F\epi A)$ to $F^e,$ then for any $F^e$-representation $\mathcal H$ of the category ${\sf Pres}(A)$
$$\ilimit^i \ \mathcal H \otimes_{F^e}\Omega(F)=0$$ for all $i.$
\end{Lemma}
\begin{proof}
It follows from proposition \ref{additive_limit_vanish} and the fact that $\Omega(F)$ is a $F^e$-additive representation (see \cite[2.3]{Bergman_Dicks}).
\end{proof}

\begin{Lemma} Let $F$ be a quasi-free algebra, $I$ its ideal and $A=F/I.$
Then the map induced by multiplication
$$(I^n/I^{n+1})\otimes_A (I^m/I^{m+1}) \longrightarrow I^{n+m}/I^{n+m+1}$$
is an isomorphism of $A$-bimodules. In particular, $(I/I^2)^{\underset{A}\otimes n}\cong I^n/I^{n+1}.$
\end{Lemma}
\begin{proof}
It is sufficient to prove that the multiplication map induces an isomorphism $(I/I^2)^{\underset{A}\otimes n}\cong I^n/I^{n+1}.$ For any two ideals $I,J\triangleleft F$ (of any algebra $F$) there is a short exact sequence
\begin{equation}0\longrightarrow {\sf Tor}_2^F(F/I,F/J)\longrightarrow I\otimes_F J \longrightarrow IJ \longrightarrow 0
\end{equation}
where the right-hand map is the multiplication map.
Since $F$ is quasi-free, its global dimension is less than or equal to $1,$  and hence $I\otimes_F J\cong IJ$ for any ideals $I,J$ of $F.$ By induction we obtain $I^{\underset{F}\otimes n}\cong I^n.$ Applying the functor $-\otimes_F A$ to the last isomorphism, we get $I^{\underset{F}\otimes n-1}\otimes_F (I/I^2)\cong I^n/I^{n+1}.$ Finally, using the isomorphism $I\otimes_F (I/I^2)\cong (I/I^2)\otimes_A (I/I^2),$ we obtain $(I/I^2)^{\underset{A}\otimes n}\cong I^n/I^{n+1}.$
\end{proof}

\begin{Lemma}\label{lemma_limits_relation_bimodule} For any $A$-bimodule $M$ the following holds
$$\ilimit^i\ (I^n/I^{n+1})\otimes_A M=\left\{\begin{array}{l l}
\Omega^n_{\rm nc}(A)\otimes_A M, & \text{ if }\ i=n\\
0, & \text{ if }\ i\ne n
\end{array}\right. $$ for $n\geq 0$, where $\Omega^n_{\rm nc}(A)=\Omega(A)^{\underset{A}\otimes n}$ {\rm (}see \cite{Ginzburg}{\rm )}.
\end{Lemma}
\begin{proof}
Consider the left-right split exact sequence of Magnus embedding
$$I/I^2 \mono  \Omega(F)\!\!\uparrow^{A^e} \epi \Omega(A).$$
Tensoring by $(I^{n-1}/I^n)\otimes_A M$, we obtain the following
short exact sequence $$0 \longrightarrow (I^n/I^{n+1})\otimes_A M
\longrightarrow  \Omega(F)\!\!\uparrow^{A^e}\otimes_A
(I^{n-1}/I^n)\otimes_A M\longrightarrow \Omega(A)\otimes_A
(I^{n-1}/I^n)\otimes_A M\longrightarrow 0.$$ If we denote
$L:=(I^{n-1}/I^n)\otimes_A M,$ the middle term can be written as
follows $$(A \otimes_{F} \Omega(F)\otimes_F A\big)\otimes_A L=A
\otimes_{F} \Omega(F)\otimes_F  L= (L\otimes A)\otimes_{F^e}
\Omega(F).$$ By lemma \ref{omega_vanish_hiher_limits} all the
higher limits of the middle term vanish.  Then using the long
exact sequence and the fact that $\Omega(A)$ is a projective as a
left $A$-module, we obtain
$$\ilimit^i ((I^n/I^{n+1})\otimes_A M)=\ilimit^{i-1} (\Omega(A)\otimes_A (I^{n-1}/I^n)\otimes_A M)= \Omega(A)\otimes_A \ilimit^{i-1}  \big((I^{n-1}/I^n)\otimes_A M\big).$$ Then by induction we get the claimed statement.
\end{proof}

\begin{Theorem}\label{hoshc}
For an algebra $A$ and an $A$-bimodule $M$, for $n\geq 1$,
there are natural isomorphism
$$
H_{2n-i}(A,M)\simeq\ilimit^i\  (I^n/I^{n+1})\otimes_{A^e} M \ \text{ for } \  i<n.
$$
\end{Theorem}
\begin{Corollary} There is an isomorphism
$$HH_{2n-i}(A)\simeq \ilimit^i\  (I^n/I^{n+1})_\natural  \ \text{ for } \  i<n.$$
\end{Corollary}
\begin{proof}
If $P$ is a projective $A$-bimodule and $M$ and $N$ are left-right projective $A$-bimodules then $P\otimes_A M$ is a projective $A$-bimodule and $M\otimes_A N$ is left-right projective. It follows from the fact that the functors  ${\rm Hom}_{A^e}(P\otimes_A M,-)\cong {\rm Hom}_{A^e}(P,{\rm Hom}_A(M,-))$, ${\rm Hom}_A(M\otimes_A N,-)\cong {\rm Hom}_A(M,{\rm Hom}_A(N,-))$ and ${\rm Hom}_{A^{op}}(M\otimes_A N,-)\cong {\rm Hom}_{A^{op}}(N,{\rm Hom}_{A^{op}}(M,-))$ are exact.

 The Magnus embedding gives the following exact sequence of representations
$I/I^2\mono \Omega(F)\!\uparrow^{A^e} \to  A^e \epi A.$  Then tensoring by $I^n/I^{n+1}$ we obtain exact sequences
$$0\longrightarrow I^{n+1}/I^{n+2} \longrightarrow \Omega(F)\!\uparrow^{A^e}\otimes_A (I^n/I^{n+1}) \longrightarrow A^e\otimes_A (I^n/I^{n+1}) \longrightarrow I^n/I^{n+1}\longrightarrow 0.$$
Consider the projective resolution $P_\bullet$ of the bimodule $A$ given by the infinite Yoneda product of these sequences. Then $P_{2n}=A^e\otimes_A (I^n/I^{n+1})$ and $P_{2n+1}=\Omega(F)\!\!\uparrow^{A^e}\otimes_A (I^n/I^{n+1})$ and there is an embedding $I^n/I^{n+1}\mono P_{2n-1}.$  Since $P_\bullet$ is a projective resolution, the homology of the complex $P_\bullet \otimes_{A^e} M$ coincides with $H_i(A,M).$ Let $\mathcal F^\bullet$ be the following cochain complex concentrated in non-negative degrees:
$$0 \longrightarrow (I^n/I^{n+1})\otimes_{A^e} M \longrightarrow P_{2n-1}\otimes_{A^e}M \longrightarrow P_{2n-2}\otimes_{A^e}M \longrightarrow \dots \longrightarrow P_0\otimes_{A^e}M\longrightarrow 0.$$ It is easy to see that $H^i(\mathcal F^\bullet)=H_{2n-i}(A,M)$ for $i\geq 0.$ The terms of $\mathcal F^\bullet$ are given by $\mathcal F^0=(I^n/I^{n+1})\otimes_{A}M$ and  $\mathcal F^p=P_{2n-p}\otimes_{A^e}M$ for $p>0.$  Therefore we obtain
$$\mathcal{F}^{2i}=M\otimes_A (I^{n-i}/I^{n-i+1}) \ \ \text{ and } \ \ \ \mathcal F^{2i-1}=\Big((I^{n-i}/I^{n-i+1})\otimes_A M\Big)\otimes_{F^e} \Omega(F)$$ for $i>0.$ Consider the spectral sequence $E$ from corollary \ref{corollary_spectral} associated with the complex $\mathcal F^\bullet$. Its first page is $E_1^{pq}=\ilimit^q \mathcal F^p,$ and $E\Rightarrow H_{2n-i}(A,M).$ The functor $\Omega(F)$ is an additive $\mathcal O$-representation, where $\mathcal O(F\epi A)=F^e$ (see \cite[2.3]{Bergman_Dicks}).  Using proposition \ref{additive_limit_vanish} and the previous lemma we get
$$E^{pq}_1=\ilimit^q \mathcal F^p=\left\{\begin{array}{l l}
0, & \ \text{ if } \ p\ne 2(n-q) \ \text{ and } p\ne 0 \\
\ilimit^q \ (I^n/I^{n+1}) \otimes_{A^e} M, & \ \text{ if }\ p=0\\
 \Omega^p_{\rm nc}(A)\otimes_A M &\ \text{ if } p=2(n-q)>0
\end{array}\right.$$
Since for any $i\leq n$ on the diagonal $p+q=i$ on $E_1$ all nonzero terms concentrated in the column $p=0$ and equal to $\ilimit^i \ (I^n/I^{n+1})\otimes_{A^e} M,$ we obtain the required isomorphisms.
 \end{proof}

\subsection{Cyclic homology of augmented algebras.}

\begin{Lemma} Let $A$ be an algebra. Then the following holds.
\begin{equation}
\ilimit^i \ F/I^2=\left\{\begin{array}{l l}
k, & i=0\\
\bar A^{\otimes 2}, & i=1\\
0, & i>1.
\end{array}\right.
\end{equation}
\end{Lemma}

\begin{proof} Note that $\Omega(F)\!\!\uparrow^{A^e}\cong \Omega(F)/(I\Omega(F)+\Omega(F)I)$ and the composition of the universal derivation and the projection $F\overset{\delta}\to \Omega(F)\epi \Omega(F)\!\!\uparrow^{A^e}$ vanishes on $I^2$. Hence the universal derivation induces the map $\hat \delta:F/I^2\to \Omega(F)\!\!\uparrow^{A^e}$ so that $\hat \delta|_{I/I^2}$ coincides with Magnus embedding $\tilde \delta:I/I^2\mono \Omega(F)\!\!\uparrow^{A^e}.$ Then the commutative  triangle
$$\xymatrix{
 & I/I^2\ar[dl]\ar[dr]^{\tilde \delta} & \\
F/I^2 \ar[rr]^{\hat \delta} && \Omega(F)\!\!\uparrow^{A^e} }$$ and
short exact sequences $ I/I^2\mono \Omega(F)\!\!\uparrow^{A^e}\epi
\Omega(A),$ $I/I^2\mono F/I^2\epi A$ give an octahedron in the
derived category $\mathbf{D}^+({\sf Vect}_k^{{\sf Pres}(A)}):$
$$\xymatrix{
& I/I^2\ar[d] \ar@{=}[r] & I/I^2\ar[d]^{\tilde \delta}  & \\
\Omega^{[0,1]}(A) \ar@{..>}[r]\ar@{=}[d] & F/I^2 \ar[r]^{\hat \delta}\ar[d] & \Omega(F)\!\!\uparrow^{A^e}\ar@{..>}[r]\ar[d] &  \Omega^{[0,1]}(A)[1]\ar@{=}[d]  \\
\Omega^{[0,1]}(A) \ar[r] & A\ar[d]\ar[r]^{\delta} & \Omega(A)\ar[r]\ar[d] & \Omega^{[0,1]}(A)[1] \\
& I/I^2[1]\ar@{=}[r] & I/I^2[1], & }$$ where $\Omega^{[0,1]}(A)$
is the complex $\cdots \to 0\to A\overset{\delta} \to \Omega(A)\to
0 \to \cdots .$ By lemma \ref{omega_vanish_hiher_limits}, we have
$\ilimit^i \ \Omega(F)\!\!\uparrow^{A^e}=0,$ and hence ${\bf
R}\!\!\ilimit \ \Omega(F)\!\!\uparrow^{A^e}=0.$
 Then, applying the total derived functor ${\bf R}\!\!\ilimit $ to the distinguished triangle $\Omega^{[0,1]}(A)\to F/I^2\to \Omega(F)\!\!\uparrow^{A^e}\to \Omega^{[0,1]}(A)[1],$ we obtain the isomorphism  ${\bf R}\!\!\ilimit\ F/I^2\cong \Omega^{[0,1]}(A)$ in the derived category ${\bf D}^+({\sf Vect}_k).$ Since ${\rm ker}(\delta)=k$ and ${\rm Coker}(\delta)=\bar A^{\otimes 2}$ (see \eqref{ker_coker_delta}), we get the claimed statement.
\end{proof}

\begin{Lemma} Let $A$ be an algebra and $n\geq 2.$ Then the unit homomorphism  $\eta_F: k\to F$ induces an isomorphism
\begin{equation} \ilimit\ F/I^n\cong k.
\end{equation}
\end{Lemma}
\begin{proof}
The proof is by induction on $n.$ The base case us given by the
previous lemma.  The inductive step follows from the short exact
sequence $I^n/I^{n+1}\mono F/I^{n+1}\epi F/I^n$ and the fact that
$\ilimit\ I^n/I^{n+1}=0=\ilimit^1\ I^n/I^{n+1}$ for $n\geq 2$
(lemma \ref{lemma_limits_relation_bimodule}).
\end{proof}

 An {\it augmented algebra} is a $k$-algebra together with an algebra homomorphism $\varepsilon_A: A\to k$ called {\it augmentation map}. The augmentation map provides a natural structure of $A$-module on $k.$  The augmentation ideal is denoted by  $\Delta(A):={\rm ker}(\varepsilon_A).$ The unit homomorphism $\eta_A:k\to A$ and the augmentation map $\varepsilon_A:A\to k$ give a natural decomposition $A=k\oplus \Delta(A).$
  If $\pi : F\epi A$ is a presentation of an augmented algebra, then the free algebra $F$ has a natural augmentation map $\varepsilon_F= \varepsilon_A\circ\pi.$

\begin{Lemma}\label{lemma_limits_In}
Let $A$ be an augmented algebra and $n\geq 1.$ Then there are
isomorphisms
\begin{equation}\ilimit^i\ \Delta(F)=0,\ \ \  \ilimit^{i} \  I^n=\ilimit^{i-1} \ \Delta(F)/I^n
\end{equation} for all $i.$
In particular, $\ilimit \ I^n=0.$
\end{Lemma}
\begin{proof}
It is easy to see that $\Delta(F)\cong \Omega(F)\otimes_F k.$ Then
the isomorphism $\Omega(F_1*F_2) \cong
\Omega(F_1)\!\!\uparrow^{(F_1*F_2)^e} \oplus\
\Omega(F_2)\!\!\uparrow^{(F_1*F_2)^e}$ implies
$\Delta(F_1*F_2)\cong \Delta(F_1)\!\!\uparrow^{F_1*F_2}\oplus \
\Delta(F_2)\!\!\uparrow^{F_1*F_2}.$ Hence the representation
$\Delta(F)$ is $F$-additive, and by proposition
\ref{additive_limit_vanish} $\ilimit^i\ \Delta(F)=\ilimit^i\
F\otimes_F\Delta(F)=0.$  This equality and the short exact
sequence  $I^n\mono \Delta(F) \epi \Delta(F)/I^n$  imply that
$\ilimit^i\ I^n = \ilimit^{i-1} \ \Delta(F)/I^n.$
\end{proof}

\begin{Theorem}\label{cychomlim}
Let $A$ be an augmented algebra over a field $k$ of characteristic
$0$. Then there are isomorphisms
\begin{equation}HC_{2n}(A) \cong \ilimit \ F/(I^{n+1}+[F,F]),
\end{equation}
\begin{equation}HC_{2n-1}(A) \cong \ilimit^1 \ F/(I^{n+1}+[F,F]),
\end{equation}
\begin{equation} \overline{HC}_{2n+1}(A)=\ilimit \ I^{n+1}/[I,I^n],
\end{equation}
\begin{equation} \overline{HC}_{2n}(A)=\ilimit^1 \ I^{n+1}/[I,I^n],\end{equation}
\end{Theorem}
\begin{Remark}
The first and the third isomorphism was proved by Quillen but we
prove it for completeness.
\end{Remark}
\begin{proof}
First of all, we prove that $\ilimit^i \ H_1(F,F/I^n)=0$ for
$i=0,1.$ Consider the natural projection $p_F:F/I^n\epi
F/(I^n+[F,F])=(F/I^n)_\natural.$ The unit homomorphism
$\eta_F:k\to F$ induces the map $\tilde\eta:k\to F/I^n$, the
augmentation map  $\varepsilon_F : F\to k$ induces the map
$\tilde\varepsilon:(F/I^n)_\natural \epi k$ and the equation
$\tilde \varepsilon \circ p_F\circ \tilde \eta={\sf id}_k$ holds.
Then we obtain the equation $\ilimit(\tilde\varepsilon)\circ
\ilimit(p_F)\circ \ilimit(\tilde \eta)={\sf id}_k,$ and hence
$\ilimit(p_F):\ilimit \ F/I^n\cong k \to \ilimit \
(F/I^n)_\natural$ is a monomorphism. Divide the exact sequence
\begin{equation}0\longrightarrow H_1(F,F/I^n)\longrightarrow \Omega(F)\otimes_{F^e}(F/I^n)\longrightarrow F/I^n \overset{p_F}\longrightarrow (F/I^n)_\natural\longrightarrow 0
\end{equation}
into two short exact sequences $\mathcal S\mono F/I^n \epi
(F/I^n)_\natural$ and $H_1(F,F/I^n)\mono
\Omega(F)\otimes_{F^e}(F/I^n) \epi \mathcal S,$ where $\mathcal S$
is some representation of ${\sf Pres}(A).$ Since $\ilimit \
p_F:\ilimit\ F/I^n \to\ilimit \ (F/I^n)_\natural$ is a
monomorphism, the first short exact sequence show that $\ilimit \
\mathcal S=0.$ By lemma \ref{omega_vanish_hiher_limits} we get
$\ilimit^i\ \Omega(F)\otimes_{F^e} (F/I^n)=0$ for all $i.$ These
equalities and the second short exact sequence imply that
$\ilimit^i\ H_1(F,F/I^n)=0$ for $i=0,1.$

Consider the Quillen's exact sequence \eqref{Quillen_sequense}
\begin{equation}0\longrightarrow HC_{2n}(A)\longrightarrow (F/I^{n+1})_\natural \longrightarrow H_1(F,F/I^n)_\sigma\longrightarrow HC_{2n-1}(A)\longrightarrow 0\end{equation}
and divide it into two short exact sequences $\mathcal T\mono
H_1(F,F/I^n)_\sigma\epi HC_{2n-1}(A)$ and $HC_{2n}(A)\mono
(F/I^{n+1})_\natural\epi \mathcal T.$ Since the field $k$ is of
characteristic $0$, the vector space $H_1(F,F/I^n)_\sigma$ is a
functorial direct summand of   $H_1(F,F/I^n)$. It follows that
$\ilimit^i\ H_1(F,F/I^n)_\sigma=0$ for $i\in \{0,1,\},$ and hence,
$\ilimit\ \mathcal T=0$ and $\ilimit^1\ \mathcal T=HC_{2n-1}(A).$
Using these isomorphisms, the second short exact sequence and the
equality $\ilimit^1\  HC_{2n}(A)=0,$ we obtain $HC_{2n}(A)=\ilimit
\ F/(I^{n+1}+[F,F])$ and $HC_{2n-1}(A)=\ilimit^1 \
F/(I^{n+1}+[F,F])$.

By the similar argument the exact sequence
\begin{equation}0\longrightarrow H_1(F,I^n)\longrightarrow \Omega(F)\otimes_{F^e} I^n \longrightarrow I^n \longrightarrow (I^n)_\natural \longrightarrow 0
\end{equation} and isomorphisms $\ilimit\ I^n=0$ (lemma \ref{lemma_limits_In}) imply $\ilimit^i \ H_1(F,I^n)=0$ for $i\in\{0,1\}.$ Finally,
using the second Quillen's sequence \eqref{Quillen_sequense2}, one can similarly prove the rest two isomorphisms.

\end{proof}

\begin{Corollary}
Let $A$ be a commutative algebra of finite type  over
$\mathbb{C}.$ Then there are isomorphisms
\begin{equation}\label{klq1}
HC_{2n}(A)= \ilimit \ F/(I^{n+1}+[F,F]),
\end{equation}
\begin{equation}
HC_{2n-1}(A)=\ilimit^1 \ F/(I^{n+1}+[F,F]).
\end{equation}
\begin{equation}\label{klq2}
\overline{HC}_{2n+1}(A)=\ilimit\ I^{n+1}/[I,I^n]
\end{equation}
\begin{equation}
\overline{HC}_{2n}(A)=\ilimit^1\ I^{n+1}/[I,I^n].
\end{equation}
\end{Corollary}
\begin{proof}
The isomorphisms (\ref{klq1}) and (\ref{klq2}) are proved by
Quillen. The others follow from the previous theorem and the fact
that any maximal ideal $\mathfrak{m}\triangleleft A$ gives an
augmentation $A\epi A/\mathfrak{m}\cong \mathbb{C}$.
\end{proof}

\section{Derived functors in the sense of Dold-Puppe as higher
limits}\label{section_Derived_functors_in_the_sense_of_Dold-Puppe_as_higher_limits} Let $T$ be an endofunctor on the category of abelian
groups. The family $L_iT(-)$ of derived functors, in the sense of
Dold-Puppe \cite{DP}, of $T$
  are defined by
$$
L_iT(A)=\pi_iTN^{-1}P_\ast,\ i\geq 0,\ A\in {\sf abelian\ groups},
$$
where $P_\ast\to A$ is a projective resolution of $A$, and
 $N^{-1}$ is  the Dold-Kan transform, which is the the inverse of the Moore normalization  functor
$$
N:  \mathcal{S}({\sf abelian\ groups}) \to \mathcal{C}h({\sf
abelian\ groups})
$$
from the category of simplicial abelian groups to the category of
chain complexes (see, for example \cite{W}, Section 8.4).

We will use the standard notation of the tensor, symmetric,
exterior and divided powers:
$$
\otimes^n, S^n,\Lambda^n,\Gamma^n: {\sf abelian\ groups}\to {\sf
abelian\ groups},\ \ n\geq 1
$$

For an abelian group $A$, we consider the category ${\sf Pres}(A)$
of presentations $H\hookrightarrow F\twoheadrightarrow A$ with a
free abelian group $F$.

\begin{Theorem}\label{derivedth}
For $n\geq 1$ and $i\geq 0$, there are natural isomorphisms
\begin{align}
& \ilimit^i\ \Lambda^n(H)=L_{n-i}S^n(A)\label{slim}\\
& \ilimit^i\ \Gamma^n(H)=L_{n-i}\Lambda^n(A)\label{glim}\\
& \ilimit^i\ \otimes^n(H)=L_{n-i}\otimes^n(A)\label{tlim}
\end{align}
where the derived limits are taken over the category ${\sf
Pres}(A)$.
\end{Theorem}

\begin{Lemma}\label{addlemma}
Let $\mathcal {F, G}$ be functors in the category of abelian
groups, which can be presented as tensor products of certain
symmetric, exterior and divided powers. Then, for any $i\geq 0$,
$$
\ilimit^i\ \left({\mathcal F}(H)\otimes {\mathcal G}(F)\right)=0.
$$
\end{Lemma}
\begin{proof}
The functors ${\mathcal F}(H)\otimes {\mathcal G}(F)$ are
monoadditive, therefore, the statement follows for $i=0$.  Observe
that the cokernel of the natural map
$$
(\mathcal F(H)\otimes {\mathcal G}(F))^2\hookrightarrow {\mathcal
F}(H\oplus F)\otimes {\mathcal G}(F\oplus F)
$$
also can be presented as a directs sum of functors of the type
considered in lemma. The statement now follows by induction on
$i$, using proposition \ref{proposition_cokernel_limits}.
\end{proof}
\noindent{\it Proof of theorem \ref{derivedth}.} Given an element
 $$ 0\to H\to F\to A\to 0$$
 of ${\mathcal Ext}(A)$,
the Koszul complexes
\begin{equation}\label{kos1}
Kos(H\to F):\ \ \  0\to \Lambda^n(H)\to \Lambda^{n-1}(H)\otimes
F\to \dots \to H\otimes S^{n-1}(F)\to S^n(F)
\end{equation}
and
\begin{equation}
Kos'(H\to F):\ \ \  0\to \Gamma^n(H)\to \Gamma^{n-1}(H)\otimes
F\to \dots \to H\otimes \Lambda^{n-1}(F)\to \Lambda^n(F)
\end{equation}
represent models of the objects $LS^n(A)$ and $L\Lambda^n(A)$ in
the derived category (see \cite{Kock}, Proposition 2.4 and Remark
2.7). Here the maps between the terms of Koszul complexes are
standard. In particular, there are natural isomorphism
$$
L_iSP^n(A)=H_i(Kos(H\to F)),\ \ \ L_i\Lambda^n(A)=H_i(Kos'(H\to
F)).
$$
Observe that, for any polynomial functor $\mathcal F$, such that
$\mathcal F(0)=0$, $\ilimit^i\ \mathcal F(F)=0$. This follows by
induction on $i$, using proposition
\ref{proposition_cokernel_limits}. The isomorphsims (\ref{slim})
and (\ref{glim}) follow from lemma \ref{addlemma} and corollary
\ref{corollary_spectral}.

It turns out from the Eilenberg-Zilber theorem that the derived
functors of the $n$-th tensor power can be described as
$$
L_i\otimes^n(A)=H_i(\underbrace{A\buildrel{L}\over\otimes \dots
\buildrel{L}\over\otimes A}_{n\ \text{terms}})=H_i((H\to
F)^{\otimes n}),\ 0\leq i\leq n-1.
$$
Now observe that, all terms of the complex $(H\to F)^{\otimes n}$
in degrees $0<i<n$ are direct sums of the functors from lemma
\ref{addlemma} and hence
$$
\ilimit^j\ ((H\to F)^{\otimes n})_i=0,\ j\geq 0,\ 0<i<n.
$$
Applying proposition \ref{proposition_cokernel_limits}, we also
get
$$
\ilimit^j\ ((H\to F)^{\otimes n})_0=\ilimit^j\ \otimes^n(F)=0,\
j\geq 0.
$$
The isomorphism (\ref{tlim}) follows from corollary
\ref{corollary_spectral} applied to the complex $(H\to F)^{\otimes
n}$. $\Box$

\vspace{.3cm} Recall that, for an abelian group $H$, there are
natural isomorphisms
$$
L_nS^n(H,1)=\Lambda^n(H),\ L_n\Lambda^n(H,1)=\Gamma^n(H),
L_n\otimes^n(H,1)=\otimes^n(H).
$$
Therefore, the formulas (\ref{slim})-(\ref{tlim}) can be written
as
\begin{equation}\label{lastfm}
\ilimit^i L_nF(H,1)=L_{n-i}F(A)
\end{equation}
where $F=S^n,\Lambda^n,\otimes^n$. It is natural to conjecture
that the formulas like (\ref{lastfm}) take place for more general
class of polynomial functors of degree $n$, for example, for
strictly $n$-polynomial functors. It is easy to show using the
above methods that (\ref{lastfm}) is true for iterated tensor
products of symmetric and exterior powers.

\end{document}